\documentclass[11pt]{amsart}

\usepackage[T1]{fontenc}
\usepackage[utf8]{inputenc}
\usepackage{lmodern}
\usepackage{microtype}
\usepackage{amssymb,amsfonts,mathtools}
\usepackage{mathrsfs}
\usepackage{enumitem}
\usepackage{hyperref}
\usepackage[nameinlink,noabbrev]{cleveref}

\hypersetup{
  colorlinks=true,
  linkcolor=blue,
  citecolor=blue,
  urlcolor=blue
}

\numberwithin{equation}{section}

\newtheorem{theorem}{Theorem}[section]
\newtheorem{proposition}[theorem]{Proposition}
\newtheorem{lemma}[theorem]{Lemma}

\theoremstyle{definition}
\newtheorem{definition}[theorem]{Definition}

\theoremstyle{remark}
\newtheorem{remark}[theorem]{Remark}

\newcommand{\RR}{\mathbb{R}}
\newcommand{\TT}{\mathbb{T}}
\newcommand{\ZZ}{\mathbb{Z}}
\newcommand{\NN}{\mathbb{N}}
\newcommand{\PP}{\mathbb{P}}
\newcommand{\EE}{\mathbb{E}}
\newcommand{\SL}{\mathrm{SL}}

\newcommand{\Prob}{\mathrm{Prob}}
\newcommand{\supp}{\mathrm{supp}}

\newcommand{\coc}{\mathcal{G}}

\newcommand{\cF}{\mathcal{F}}

\newcommand{\cM}{\mathcal{M}}

\begin{document}

\title[Lyapunov positivity and monodromy criteria]
{Lyapunov Positivity and Monodromy Criteria for Moving-Law Random--Quasiperiodic Cocycles}

\author[Z. Liu]{Zhongmin Liu}
\address{College of Science, National University of Defense Technology, Changsha, 410073, PR China}
\email{1254642090@qq.com}

\author[D. Chen]{Dengdi Chen}
\address{College of Science, National University of Defense Technology, Changsha, 410073, PR China}
\email{dengdichen@yeah.net}

\author[Y. Zheng]{Yan Zheng}
\address{College of Science, National University of Defense Technology, Changsha, 410073, PR China}
\email{yanzhengyl@163.com}
\thanks{Corresponding author: Yan Zheng.}
\thanks{This work was supported by the National Natural Science Foundation of China under Grant 12571189.}

\date{\today}

\begin{abstract}
We study moving-law random-environment matrix cocycles whose one-step
quasiperiodic cocycle law varies stationarily along an ergodic environment
orbit.  Two fixed-law structures then cease to suffice: a single stationary
projective measure does not record transport between successive environment
fibers, and pooling the moving supports may create products that are not
dynamically admissible along any environment orbit.  We replace the first by
environment-indexed stationary projective families and prove an
environment-wise Furstenberg formula.  When the top Lyapunov exponent
vanishes, averaged stationarity upgrades to exact equivariance at the support
level,
\(g_*\eta^\omega=\eta^{\Theta\omega}\), which forces either an
environment-wise compact reduction or a finite support-invariant family of
nonzero proper measurable subbundles.  Environment-wise noncompactness and
environment-wise strong irreducibility therefore imply positivity.  To
address the second obstruction, we compare admissible products along matched
environment segments with equal total torus translation. Thus  their quotients are
genuine returns on a common vector fiber and generate environment-matched
monodromy groups.  Stable noncompactness and stable strong irreducibility of
these groups verify the cocycle-level hypotheses.  Applied to stationarily
varying random--quasiperiodic Schr\"odinger cocycles, the framework yields
strict positivity of the top Lyapunov exponent at every real energy under the
stated assumptions.
\end{abstract}

\maketitle

\section{Introduction}

Lyapunov exponents quantify the asymptotic exponential growth rates of linear
cocycles.  A positive top Lyapunov exponent means that typical long products
expand some directions exponentially.  Criteria for this positivity are
therefore central in random matrix products, linear cocycles, random dynamical
systems, and the transfer-matrix analysis of one-dimensional Schr\"odinger
operators \cite{FK60,Oseledets68,Viana14,BL85}.

For products generated by a single probability law, Furstenberg's theory
relates positivity of the top Lyapunov exponent to stationary measures on projective space and to the
geometry of the group generated by the support \cite{Fur63,FK83,BL85}.
Cai, Duarte, and Klein (CDK) extended this framework to fixed-law mixed
random--quasiperiodic cocycles \cite{CDK22,CDK23}.  In their model, a
quasiperiodic cocycle element \(g=(\alpha,A)\) shifts the torus phase by
\(\alpha\) and acts on the vector fiber through the phase-dependent matrix
\(A\), and the cocycle elements are sampled independently across time from a
single compactly supported law \(\nu\).  They prove a Furstenberg-type
positivity criterion under ergodicity, noncompactness, and strong
irreducibility, and introduce a monodromy mechanism that provides more readily
verifiable sufficient conditions for the latter two hypotheses.

Discrete Schr\"odinger operators provide a principal motivation for CDK's
mixed random--quasiperiodic framework.  CDK apply the preceding
Furstenberg--monodromy mechanism to the associated dynamics
\[
        \theta_{n+1}=\theta_n+\alpha_n,
        \qquad
        A_{E,\zeta_n}(\theta_n)
        =
        \begin{pmatrix}
        v(\theta_n)+\zeta_n-E&-1\\
        1&0
        \end{pmatrix}.
\]
Here
\(\theta_n\) is the phase, \(\alpha_n\) the frequency or torus translation,
\(v\) the quasiperiodic background potential, \(\zeta_n\) the random
perturbation, and \(E\) the energy.  In CDK's model, successive
frequency--noise pairs \((\alpha_n,\zeta_n)\) are sampled independently across
time from a single fixed joint law on \(\mathbb T^d\times\mathbb R\).
Products of the above matrices propagate the two-component solution data of the
associated Schr\"odinger eigenvalue equation, providing the standard
connection between cocycle growth and the spectral theory of one-dimensional
Schr\"odinger operators.

CDK's fixed-law i.i.d. framework, however, does not cover situations in which
an external ergodic environment changes the admissible one-step law over time.
Environment-driven temporal variation has precedents in random dynamical
systems and in work on random environments, Markov-dependent products, and
independent non-identically distributed matrix products
\cite{Arnold98,Kif01,Goldsheid22,GK26}.  In a finite-state Markov mixed
random--quasiperiodic setting, Tall \cite{Tall26} likewise retains the current
Markov state through a finite family of conditional projective measures while
studying the Lyapunov exponent as a function of transition probabilities,
with the cocycle maps, translations, and support graph fixed.  These works
treat important forms of temporal dependence.  Here we consider a mixed
random--quasiperiodic model in which a general ergodic environment selects the
one-step law at each time and the corresponding supports may vary along the
environment orbit.

We now introduce the moving-law random-environment setting.  Here and
throughout, ``moving-law'' means stationary variation driven by the ergodic
environment, not arbitrary deterministic nonstationarity.  Let
\((\Omega,\cF,\PP,\Theta)\) be an invertible, ergodic,
probability-preserving environment system.  Each state \(\omega\) determines
a one-step cocycle law \(\nu^\omega\) on the quasiperiodic cocycle group
\(G=\coc(d,r)\), so time \(n\) uses \(\nu^{\Theta^n\omega}\).  Conditioned
on the initial state \(\omega\), the path law is
\[
        \mathbf P^\omega
        =
        \bigotimes_{n\ge0}\nu^{\Theta^n\omega}.
\]
Thus the environment orbit determines the one-step law at each time, and
\(\mathbf P^\omega\) governs the resulting random sequence of 
quasiperiodic cocycle elements whose time-ordered compositions generate
the moving-law random-environment matrix cocycle.  In the
Schr\"odinger specialization, the measurable family
\(\omega\mapsto\kappa^\omega\in\Prob_c(\TT^d\times\RR)\) assigns to each
environment the joint law of the frequency--noise pair.  Conditioned on the
initial environment  \(\omega\), the pair at time \(n\) is sampled from
\(\kappa^{\Theta^n\omega}\), whose push-forward under
\((\alpha,\zeta)\mapsto(\alpha,A_{E,\zeta})\) is the energy-dependent
one-step cocycle law \(\nu_E^{\Theta^n\omega}\).  Thus CDK's
single fixed joint law is replaced by an environment-indexed family varying
stationarily along the environment orbit.  This variation affects both the
projective stationary structure and the support products that can occur along
the dynamics.

At the projective level, a single stationary measure cannot record how
projective data pass from the \(\omega\)-fiber to the
\(\Theta\omega\)-fiber.  The appropriate object is a measurable family
\(\eta^\bullet=\{\eta^\omega\}_{\omega\in\Omega}\) satisfying averaged stationarity:
\[
        \eta^{\Theta\omega}
        =
        \int_G g_*\eta^\omega\,d\nu^\omega(g)
        \qquad\text{for }\PP\text{-a.e. }\omega .
\]
We prove an environment-wise Furstenberg formula that recovers the variational
characterization of the top Lyapunov exponent over this family domain.  When
the exponent vanishes, however, the formula only identifies the variational maximum
as zero.  A separate regularization-and-rigidity argument produces a
stationary projective family and a full-measure environment set on which
\(g_*\eta^\omega=\eta^{\Theta\omega}\) for every
\(g\in\supp\nu^\omega\).  Thus CDK's invariance of one measure under the group
generated by a fixed support is replaced by exact transport across successive
environment fibers.  This exact equivariance forces at least one of two
degeneracies: an environment-wise compact reduction or a finite
support-invariant family of nonzero proper measurable subbundles.  Under the
standing mixed-base ergodicity and logarithmic-integrability assumptions,
environment-wise noncompactness and environment-wise strong irreducibility
exclude these alternatives, respectively, and yield the environment-wise positivity criterion of Theorem~\ref{thm:env-positivity}.  

At the support level, the admissible cocycle factor at time \(j\) is
constrained by \(\supp\nu^{\Theta^j\omega}\), and  pooling the moving supports may
therefore create formal products that occur along no admissible environment
orbit.  We compare only products drawn along the same forward or backward
environment segment, so they have the same environment endpoints and follow
the same ordered sequence of supports.  We additionally require equal total
torus translation, so their phase endpoints agree.  The relative quotient is
then a genuine return on a common vector fiber.  These returns generate the
environment-matched equal-total monodromy groups.  Stable noncompactness and
stable strong irreducibility of the resulting groups provide the cocycle-level
conditions required by Theorem~\ref{thm:env-positivity}.  Applied to
stationarily varying random--quasiperiodic Schr\"odinger cocycles, the current
and preceding environments supply lower- and upper-unipotent returns,
respectively, and environment matching places these witnesses in the same
environment-matched equal-total monodromy group.  The criterion then yields strict
positivity of the top Lyapunov exponent at every real energy under the stated
hypotheses.  In the fixed-law specialization, environment matching is
automatic, and the criterion recovers the corresponding all-energy positivity
conclusions of CDK \cite[Propositions~4.1 and~4.4]{CDK23} as special cases,
thereby extending those positivity results to stationarily varying laws.

The remainder of the paper is organized as follows.  \Cref{sec:model} develops
the moving-law random-environment matrix cocycle, the translation factor,
mixed-base ergodicity, and the top Lyapunov exponent.
\Cref{sec:stationary-families} treats stationary projective families, the Haar
torus marginal, and the environment-wise Furstenberg formula.
\Cref{sec:exact-equivariance} proves exact equivariance, the degeneracy
alternative, and the environment-wise positivity criterion.
\Cref{sec:return-monodromy} constructs environment-matched equal-total word
quotients and monodromy groups and proves the monodromy exclusion and
positivity criteria.  \Cref{sec:schrodinger} gives the stationarily varying
random--quasiperiodic Schr\"odinger application.

\section{Moving-Law Random-Environment Matrix Cocycle Model}
\label{sec:model}

This section gives the formal definition of the moving-law random-environment
matrix cocycle.  Each sampled cocycle element has a torus translation
component, which advances the phase, and a phase-dependent matrix component,
which acts on the vector fiber.  Their sampled sequence defines the mixed
random--quasiperiodic structure, while the one-step cocycle law changes with
the environment state.  We also isolate the translation factor, which is the only
obstruction to uniqueness of the Haar torus family and hence to the Haar torus
marginal statement used in Section~\ref{sec:stationary-families}.

\subsection{Cocycle group, conditional path laws, and the mixed base}

Throughout the paper, \((\Omega,\cF,\PP)\) is a standard Borel probability
space, and
\(\Theta:\Omega\to\Omega\)
is a bimeasurable invertible probability-preserving transformation.  Unless
explicitly stated otherwise, we assume environment ergodicity, meaning that
\(\Theta\) is ergodic.  

Set
\(
        G:=\coc(d,r)=\TT^d\times C^0(\TT^d,\SL_r(\RR)).
\)

\begin{definition}[Cocycle group and coordinates]
\label{def:cocycle-group-convention}
Equip \(C^0(\TT^d,\SL_r(\RR))\) with the uniform topology and \(G\) with the
product topology.  For \(g=(\alpha,A)\in G\), write
\(\mathfrak a(g):=\alpha, \ \mathcal A(g):=A\).
Thus
\(\mathfrak a:G\to\TT^d, \ \mathcal A:G\to C^0(\TT^d,\SL_r(\RR))\)
are the coordinate projections.

For \(g=(\alpha,A)\) and \(h=(\beta,B)\), define
\[
        g\circ h
        :=
        \bigl(\alpha+\beta,(A\circ\tau_\beta)B\bigr),
        \qquad
        \tau_\beta(\theta):=\theta+\beta .
\]
Equivalently,
\(\mathcal A(g\circ h)(\theta) = A(\theta+\beta)B(\theta)\).
The identity is \(e=(0,I)\), where \(I\) is the constant identity matrix, and
\[
        g^{-1}
        =
        \bigl(-\alpha,A^{-1}\circ\tau_{-\alpha}\bigr),
        \qquad
        A^{-1}(\theta):=A(\theta)^{-1}.
\]
With these operations, \(G\) is a Polish topological group.
\end{definition}

Let
\(\omega\longmapsto \nu^\omega\in\Prob(G)\)
be a measurable Borel probability kernel.  At environment state \(\omega\), the
one-step cocycle law is \(\nu^\omega\). After one step, the environment is
\(\Theta\omega\).  Hence the conditional path law starting from \(\omega\) is
\[
        \mathbf P^\omega
        :=
        \bigotimes_{n\ge0}\nu^{\Theta^n\omega}
\]
on \(G^{\NN}\).  Since \(G\) is Polish and the one-step laws form a
measurable probability kernel, the Ionescu--Tulcea product-kernel theorem
\cite[Theorem~8.24]{Kallenberg21} gives a measurable kernel
\(\omega\mapsto\mathbf P^\omega\).

Let \(\sigma\) be the left shift on \(G^{\NN}\).  The mixed-base map is
\[
        \mathfrak F:
        \Omega\times G^{\NN}\times\TT^d
        \longrightarrow
        \Omega\times G^{\NN}\times\TT^d
\]
defined by
\(
        \mathfrak F(\omega,\xi,\theta)
        :=
        (\Theta\omega,\sigma\xi,\theta+\mathfrak a(\xi_0)).
\)
Let \(m\) denote Haar probability measure on \(\TT^d\), and define
\[
        d\rho(\omega,\xi,\theta)
        =
        d\PP(\omega)\,d\mathbf P^\omega(\xi)\,dm(\theta).
\]
Proposition~\ref{prop:mixed-base-invariance} shows that \(\rho\) is
invariant under \(\mathfrak F\).  We call the resulting measure-preserving
system the \emph{mixed base} and \(\rho\) its \emph{mixed invariant measure}.

\begin{proposition}[Mixed-base invariance]
\label{prop:mixed-base-invariance}
The map \(\mathfrak F\) preserves \(\rho\).
\end{proposition}

\begin{proof}
Let \(\Phi:\Omega\times G^{\NN}\times\TT^d\to\RR\) be bounded measurable.  Under
the canonical identification \(G^{\NN}\simeq G\times G^{\NN}\),
\(\mathbf P^\omega=\nu^\omega\otimes\mathbf P^{\Theta\omega}\).
Define
\[
        H(\omega)
        :=
        \int_{G^{\NN}}\int_{\TT^d}
        \Phi(\omega,\zeta,\theta)\,dm(\theta)\,
        d\mathbf P^\omega(\zeta).
\]
Then
\[
\begin{aligned}
\int \Phi\circ\mathfrak F\,d\rho
&=
\int_\Omega\int_G\int_{G^{\NN}}\int_{\TT^d}
\Phi(\Theta\omega,\zeta,\theta+\mathfrak a(g))\,
dm(\theta)\,d\mathbf P^{\Theta\omega}(\zeta)\,
d\nu^\omega(g)\,d\PP(\omega)\\
&=
\int_\Omega\int_G\int_{G^{\NN}}\int_{\TT^d}
\Phi(\Theta\omega,\zeta,\theta)\,
dm(\theta)\,d\mathbf P^{\Theta\omega}(\zeta)\,
d\nu^\omega(g)\,d\PP(\omega)\\
&=
\int_\Omega H(\Theta\omega)\,d\PP(\omega)
=
\int_\Omega H(\omega)\,d\PP(\omega)
=
\int \Phi\,d\rho .
\end{aligned}
\]
We used the product-kernel decomposition, translation invariance of Haar
measure, that \(\nu^\omega\) is a probability measure, and
\(\Theta_*\PP=\PP\).
\end{proof}

\subsection{Translation factor and mixed-base ergodicity}

Only the translation component of \(g\in G\) enters the base motion on the
torus.  Set
\[
        \Sigma:=\TT^d,
        \qquad
        \mu^\omega:=\mathfrak a_*\nu^\omega,
        \qquad
        \mathbf Q^\omega:=\bigotimes_{n\ge0}\mu^{\Theta^n\omega}.
\]
We call \(\mu^\omega\) the \emph{translation marginal} of the one-step cocycle
law, and \(\mathbf Q^\omega\) the \emph{conditional path law of the translation
coordinates}.
Let \(\hat\sigma\) be the left shift on \(\Sigma^{\NN}\), and define
\(\hat{\mathfrak F}(\omega,\beta,\theta) := (\Theta\omega,\hat\sigma\beta,\theta+\beta_0)\)
on \(\Omega\times\Sigma^{\NN}\times\TT^d\).  Put
\(d\hat\rho(\omega,\beta,\theta) = d\PP(\omega)\,d\mathbf Q^\omega(\beta)\,dm(\theta)\).
The factor map
\(\Pi(\omega,(\xi_n)_{n\ge0},\theta) := \bigl(\omega,(\mathfrak a(\xi_n))_{n\ge0},\theta\bigr)\)
satisfies
\(\Pi\circ\mathfrak F=\hat{\mathfrak F}\circ\Pi, \ \Pi_*\rho=\hat\rho\).
Thus
\((\Omega\times\Sigma^{\NN}\times\TT^d,\hat{\mathfrak F},\hat\rho)\)
is the translation factor of
\((\Omega\times G^{\NN}\times\TT^d,\mathfrak F,\rho)\).
We use \emph{mixed-base ergodicity} for ergodicity of the full mixed base and
\emph{translation-factor ergodicity} for ergodicity of this factor.  Neither
term means environment ergodicity, which concerns only \(\Theta\).

A measurable family
\(\lambda^\bullet=\{\lambda^\omega\}_{\omega\in\Omega}, \ \lambda^\omega\in\Prob(\TT^d)\),
is called a \emph{stationary torus family} if
\(\lambda^{\Theta\omega}=\mu^\omega*\lambda^\omega \ \text{for }\PP\text{-a.e. }\omega\).
The family \(\lambda^\omega=m\) is the \emph{Haar torus family}.

For a bounded measurable function \(\Phi\) on \(\Omega\times\TT^d\), define the
skew Markov operator
\[
        (\mathcal Q\Phi)(\omega,\theta)
        :=
        \int_{\TT^d}\Phi(\Theta\omega,\theta+\alpha)\,d\mu^\omega(\alpha).
\]
Since the Haar torus family is stationary, \(\mathcal Q\) is well defined as a
contraction on \(L^\infty(\Omega\times\TT^d,\PP\times m)\).  We call
\(\Phi\in L^\infty(\Omega\times\TT^d,\PP\times m)\)
\(\mathcal Q\)-stationary if
\(\mathcal Q\Phi=\Phi \  \text{in }L^\infty(\Omega\times\TT^d,\PP\times m)\).

For \(\beta=(\beta_j)_{j\ge0}\in\Sigma^{\NN}\), set
\[
        \Gamma_0(\beta):=0,
        \qquad
        \Gamma_n(\beta):=\sum_{j=0}^{n-1}\beta_j
        \quad(n\ge1).
\]
For \(\xi=(\xi_j)_{j\ge0}\in G^{\NN}\), set
\[
        \Gamma_0(\xi):=0,
        \qquad
        \Gamma_n(\xi):=\sum_{j=0}^{n-1}\mathfrak a(\xi_j)
        \quad(n\ge1).
\]
If \(\beta_j=\mathfrak a(\xi_j)\) for all \(j\), then
\(\Gamma_n(\beta)=\Gamma_n(\xi)\).  We call these quantities the
\emph{cumulative torus translations} along the corresponding path prefixes.

For \(k\in\ZZ^d\), set
\(e_k(x):=e^{2\pi i k\cdot x}, \ x\in\TT^d\).
For a finite complex Borel measure \(\zeta\) on \(\TT^d\), define
\[
        \widehat{\zeta}(k)
        :=
        \int_{\TT^d}e_k(\theta)\,d\zeta(\theta).
\]
Then
\(\widehat{\zeta_1*\zeta_2}(k) = \widehat{\zeta_1}(k)\widehat{\zeta_2}(k)\).
In particular, if \(\lambda^\bullet\) is a stationary torus family, then
\[
        \widehat{\lambda^{\Theta\omega}}(k)
        =
        \widehat{\mu^\omega}(k)\widehat{\lambda^\omega}(k)
        \qquad
        \text{for }\PP\text{-a.e. }\omega .
\]

The following criterion is the environment-wise torus version of the usual
Fourier criteria for compact-group ergodicity and random walks
\cite{EW11,Ros81}.

\begin{theorem}[Mixed-base ergodicity criterion]
\label{thm:ergodicity-random-environment}
The following statements are equivalent.
\begin{enumerate}
\item The measure-preserving system
\((\Omega\times G^{\NN}\times\TT^d,\mathfrak F,\rho)\)
is ergodic.

\item The translation factor
\((\Omega\times\Sigma^{\NN}\times\TT^d,\hat{\mathfrak F},\hat\rho)\)
is ergodic.

\item Every \(\mathcal Q\)-stationary observable
\(\Phi\in L^\infty(\Omega\times\TT^d,\PP\times m)\)
is constant \((\PP\times m)\)-a.e.

\item For every \(k\in\ZZ^d\setminus\{0\}\), the equation
\(
        u(\omega)=\widehat{\mu^\omega}(k)\,u(\Theta\omega)
        \ 
        \text{for }\PP\text{-a.e. }\omega
\)
has no nonzero solution \(u\in L^\infty(\Omega,\PP)\).

\item Every stationary torus family is the Haar torus family.
\end{enumerate}
\end{theorem}

\begin{proof}
\smallskip
\noindent
\((1)\Rightarrow(2)\) follows from the factor relation.

\smallskip
\noindent
\((2)\Rightarrow(1)\).
It suffices to prove Cesaro convergence of correlations on a dense subspace of
\(L^2(\rho)\).  Let \(\mathcal V_n\) consist of functions
\(F(\omega,\xi,\theta) = u(\omega,\xi_0,\dots,\xi_n)\varphi(\theta)\),
where \(u\) is bounded measurable and \(\varphi\in C(\TT^d)\), and set
\(\mathcal V:=\bigcup_{n\ge0}\mathcal V_n\).  The linear span of \(\mathcal V\)
is dense in \(L^2(\rho)\).

Apply the disintegration theorem
\cite[Theorem~8.5]{Kallenberg21} to the joint probability
\(\PP(d\omega)\nu^\omega(dg)\) and the Borel map
\((\omega,g)\mapsto(\omega,\mathfrak a(g))\).  Since all spaces involved are
standard Borel, there is a jointly measurable conditional kernel
\((\omega,\alpha)\mapsto\nu^\omega_\alpha\) such that
\[
        \nu^\omega
        =
        \int_{\TT^d}\nu^\omega_\alpha\,d\mu^\omega(\alpha),
        \qquad
        \nu^\omega_\alpha(\mathfrak a^{-1}\{\alpha\})=1
        \quad
        \text{for }\mu^\omega\text{-a.e. }\alpha .
\]
For \((\omega,\beta)\in\Omega\times\Sigma^{\NN}\), the product-kernel theorem
\cite[Theorem~8.24]{Kallenberg21} gives the measurable conditional path kernel
\[
        \mathbf P^{\omega,\beta}
        :=
        \bigotimes_{j\ge0}\nu^{\Theta^j\omega}_{\beta_j}.
\]
Iterating the preceding disintegration coordinate by coordinate shows that
\(\rho\) disintegrates over \(\hat\rho\) along \(\Pi\):
\begin{equation}\label{eq:rho-disintegration-random-env}
        \int H\,d\rho
        =
        \int_{\Omega\times\Sigma^{\NN}\times\TT^d}
        \left(
        \int_{G^{\NN}}H(\omega,\xi,\theta)\,
        d\mathbf P^{\omega,\beta}(\xi)
        \right)
        d\hat\rho(\omega,\beta,\theta).
\end{equation}

For bounded measurable \(u\) on \(\Omega\times G^{n+1}\), define
\[
        (\mathsf E_nu)(\omega,\beta)
        :=
        \int_{G^{n+1}}
        u(\omega,\xi_0,\dots,\xi_n)
        \prod_{j=0}^n
        \nu^{\Theta^j\omega}_{\beta_j}(d\xi_j).
\]
Let
\(T(\omega,\xi):=(\Theta\omega,\sigma\xi), \  \hat T(\omega,\beta):=(\Theta\omega,\hat\sigma\beta)\).
Product independence under \(\mathbf P^{\omega,\beta}\) gives, for \(N>n\),
\begin{equation}\label{eq:factorization-random-env}
        \mathsf E_{N+n}\bigl[(u\circ T^N)v\bigr]
        =
        (\mathsf E_nu\circ \hat T^N)(\mathsf E_nv).
\end{equation}

Take \(F,G\in\mathcal V\), enlarging \(n\) so that
\(F=u(\omega,\xi_0,\dots,\xi_n)\varphi(\theta), \ G=v(\omega,\xi_0,\dots,\xi_n)\psi(\theta)\).
Using \eqref{eq:rho-disintegration-random-env} and
\eqref{eq:factorization-random-env}, for \(N>n\),
\[
\begin{aligned}
\int (F\circ\mathfrak F^N)G\,d\rho
&=
\int_{\Omega\times\Sigma^{\NN}\times\TT^d}
(\mathsf E_nu\circ\hat T^N)(\omega,\beta)
(\mathsf E_nv)(\omega,\beta) \\
&\hspace{3.5cm}\times
\varphi(\theta+\Gamma_N(\beta))\psi(\theta)\,
d\hat\rho(\omega,\beta,\theta).
\end{aligned}
\]
Define
\[
        \widehat F(\omega,\beta,\theta)
        :=
        (\mathsf E_nu)(\omega,\beta)\varphi(\theta),
        \qquad
        \widehat G(\omega,\beta,\theta)
        :=
        (\mathsf E_nv)(\omega,\beta)\psi(\theta).
\]
Since
\(\hat{\mathfrak F}^N(\omega,\beta,\theta) = (\Theta^N\omega,\hat\sigma^N\beta,\theta+\Gamma_N(\beta))\),
we obtain
\begin{equation}\label{eq:correlation-transfer-random-env}
        \int (F\circ\mathfrak F^N)G\,d\rho
        =
        \int (\widehat F\circ\hat{\mathfrak F}^N)\widehat G\,d\hat\rho,
        \qquad N>n.
\end{equation}
Moreover,
\[
        \int \widehat F\,d\hat\rho=\int F\,d\rho,
        \qquad
        \int \widehat G\,d\hat\rho=\int G\,d\rho .
\]
By ergodicity of \((\hat{\mathfrak F},\hat\rho)\), the correlation criterion for
ergodicity, and \eqref{eq:correlation-transfer-random-env}, Cesaro correlation
convergence holds for all \(F,G\in\mathcal V\).  Density then gives ergodicity
of \((\mathfrak F,\rho)\).

\smallskip
\noindent
\((2)\Rightarrow(3)\).
Let
\(\Phi\in L^\infty(\Omega\times\TT^d,\PP\times m), \ \mathcal Q\Phi=\Phi\).
Choose a bounded representative.  For \(n\ge0\), set
\(M_n(\omega,\beta,\theta) := \Phi(\Theta^n\omega,\theta+\Gamma_n(\beta))\).
Let
\[
        \mathscr G_0:=\sigma(\omega,\theta),
        \qquad
        \mathscr G_n:=\sigma(\omega,\theta,\beta_0,\dots,\beta_{n-1})
        \quad(n\ge1).
\]
The map
\((\omega,\beta,\theta) \longmapsto (\Theta^n\omega,\theta+\Gamma_n(\beta))\)
pushes \(\hat\rho\) forward to \(\PP\times m\).  Hence \(M_n\) is bounded and
\(\mathscr G_n\)-measurable.  Since, conditionally on \(\mathscr G_n\),
\(\beta_n\) has law \(\mu^{\Theta^n\omega}\),
\[
\begin{aligned}
\EE_{\hat\rho}[M_{n+1}\mid\mathscr G_n]
&=
\int_{\TT^d}
\Phi(\Theta^{n+1}\omega,\theta+\Gamma_n(\beta)+\alpha)\,
d\mu^{\Theta^n\omega}(\alpha)\\
&=
(\mathcal Q\Phi)(\Theta^n\omega,\theta+\Gamma_n(\beta))
=
M_n .
\end{aligned}
\]
Thus \((M_n,\mathscr G_n)\) is a bounded, hence uniformly integrable,
martingale.  The martingale convergence theorems
\cite[Theorems~9.19 and~9.22]{Kallenberg21} give convergence a.e. and in
\(L^1(\hat\rho)\) to some \(M_\infty\).  Since
\(M_n\circ\hat{\mathfrak F}=M_{n+1}\),
and \(\hat{\mathfrak F}\) preserves \(\hat\rho\), the limit satisfies
\(M_\infty\circ\hat{\mathfrak F}=M_\infty, \ \hat\rho\text{-a.e.}\)
By ergodicity of the translation factor, \(M_\infty=c\) a.e.  Since
\(\EE_{\hat\rho}[M_n\mid\mathscr G_0]=M_0 \ \text{for all }n\),
passing to the \(L^1\)-limit gives
\(M_0=\EE_{\hat\rho}[M_\infty\mid\mathscr G_0]=c\).
As \(M_0(\omega,\beta,\theta)=\Phi(\omega,\theta)\), \(\Phi=c\)
\((\PP\times m)\)-a.e.

\smallskip
\noindent
\((3)\Rightarrow(2)\).
Let \(H\in L^\infty(\hat\rho)\) satisfy
\(H\circ\hat{\mathfrak F}=H \ \hat\rho\text{-a.e.}\)
Define
\[
        \Phi(\omega,\theta)
        :=
        \int_{\Sigma^{\NN}}H(\omega,\beta,\theta)\,
        d\mathbf Q^\omega(\beta).
\]
Using
\(\mathbf Q^\omega=\mu^\omega\otimes\mathbf Q^{\Theta\omega}\)
and \(H=H\circ\hat{\mathfrak F}\), we get
\(\Phi=\mathcal Q\Phi\).
By \((3)\), \(\Phi=c\) \((\PP\times m)\)-a.e.

For the filtration \((\mathscr G_n)_{n\ge0}\) above,
\(
        \EE_{\hat\rho}[H\mid\mathscr G_n](\omega,\beta,\theta)
        =
        \Phi(\Theta^n\omega,\theta+\Gamma_n(\beta)).
\)
Indeed, this follows by testing against bounded \(\mathscr G_n\)-measurable
functions, using \(H=H\circ\hat{\mathfrak F}^n\), and integrating out the tail.
The right-hand side is equal to \(c\) a.e. because the displayed push-forward
map has marginal \(\PP\times m\).  Hence
\(\EE_{\hat\rho}[H\mid\mathscr G_n]=c, \ \hat\rho\text{-a.e.}\)
for every \(n\).  Since \(\mathscr G_n\) increases to the full
\(\sigma\)-algebra, the upward theorem for conditional expectations
\cite[Theorem~9.24]{Kallenberg21} gives
\(\EE_{\hat\rho}[H\mid\mathscr G_n]\to H\) in \(L^1\), and therefore
\(H=c\) a.e.  Thus the translation factor is ergodic.

\smallskip
\noindent
\((3)\Leftrightarrow(4)\).
For
\(\Phi\in L^\infty(\Omega\times\TT^d,\PP\times m) \subset L^2(\Omega\times\TT^d,\PP\times m)\),
define
\[
        \widehat{\Phi}(\omega,k)
        :=
        \int_{\TT^d}\Phi(\omega,\theta)\overline{e_k(\theta)}\,dm(\theta).
\]
Then \(\widehat{\Phi}(\cdot,k)\in L^\infty(\Omega,\PP)\).  A direct Fourier
calculation gives
\(
        \widehat{(\mathcal Q\Phi)}(\omega,k)
        =
        \widehat{\mu^\omega}(k)\widehat{\Phi}(\Theta\omega,k).
\)
Therefore \(\mathcal Q\Phi=\Phi\) if and only if, for every \(k\in\ZZ^d\),
\(
        \widehat{\Phi}(\omega,k)
        =
        \widehat{\mu^\omega}(k)\widehat{\Phi}(\Theta\omega,k)
        \ 
        \PP\text{-a.e.}
\)

Assume \((4)\).  If \(\mathcal Q\Phi=\Phi\), then all nonzero Fourier modes
vanish by \((4)\).  For \(k=0\),
\(\widehat{\Phi}(\omega,0)=\widehat{\Phi}(\Theta\omega,0)\),
so ergodicity of \(\Theta\) makes the zero mode constant.  Hence \(\Phi\) is
constant, proving \((3)\).

Conversely, if \((4)\) fails, choose \(k_0\neq0\) and nonzero
\(u\in L^\infty(\Omega,\PP)\) such that
\(u(\omega)=\widehat{\mu^\omega}(k_0)u(\Theta\omega)\).
Then
\(\Phi(\omega,\theta):=u(\omega)e_{k_0}(\theta)\)
is nonconstant and satisfies \(\mathcal Q\Phi=\Phi\).  Hence \((3)\) fails.

\smallskip
\noindent
\((4)\Leftrightarrow(5)\).
Assume first that \((4)\) holds, and let
\(\lambda^\bullet=\{\lambda^\omega\}_{\omega\in\Omega}\)
be stationary on \(\TT^d\).  For \(k\in\ZZ^d\), put
\(u_k(\omega):=\widehat{\lambda^\omega}(k), \ a_k(\omega):=\widehat{\mu^\omega}(k)\).
Stationarity gives
\(u_k(\Theta\omega)=a_k(\omega)u_k(\omega), \ \PP\text{-a.e.}\)
Suppose that \(\lambda^\bullet\) is not the Haar torus family.  Then for some
\(k_0\neq0\), \(u:=u_{k_0}\) is nonzero on a set of positive measure.  Writing
\(a:=a_{k_0}\), we have
\(u(\Theta\omega)=a(\omega)u(\omega), \ |a(\omega)|\le1\).
Thus \(|u\circ\Theta|\le |u|\).  Integration and \(\Theta_*\PP=\PP\) imply
equality a.e.; by ergodicity, \(u\neq0\) a.e. and \(|a|=1\) a.e.  Hence
\[
        u(\omega)=\overline{a(\omega)}\,u(\Theta\omega)
        =
        \widehat{\mu^\omega}(-k_0)u(\Theta\omega),
\]
which is a nonzero bounded solution forbidden by \((4)\) for the mode
\(-k_0\).  Therefore all nonzero Fourier--Stieltjes coefficients of
\(\lambda^\omega\) vanish for \(\PP\)-a.e. \(\omega\).  By countability of
\(\ZZ^d\setminus\{0\}\) and uniqueness of Fourier--Stieltjes coefficients,
\(\lambda^\omega=m\) for \(\PP\)-a.e. \(\omega\).

Conversely, assume \((5)\) and suppose that \((4)\) fails.  Then for some
\(k_0\neq0\) there is a nonzero \(u\in L^\infty(\Omega,\PP)\) satisfying
\(u(\omega)=a(\omega)u(\Theta\omega), \ a(\omega):=\widehat{\mu^\omega}(k_0)\).
As above, \(|a|=1\) a.e. and \(u\neq0\) a.e.  Define
\[
        \zeta(\omega):=\frac{u(\omega)}{|u(\omega)|}
\]
on the full-measure set where \(u\neq0\), and define it arbitrarily on the
remaining null set.  Then \(|\zeta|=1\) and
\(\zeta(\Theta\omega)=\overline{a(\omega)}\,\zeta(\omega), \ \PP\text{-a.e.}\)

Fix \(0<\varepsilon<1/2\), and define
\[
        d\lambda^\omega(\theta)
        =
        \left(
        1+
        \varepsilon\zeta(\omega)e_{k_0}(\theta)
        +
        \varepsilon\overline{\zeta(\omega)}e_{-k_0}(\theta)
        \right)dm(\theta).
\]
Equivalently, the density is
\(1+2\varepsilon\Re\bigl(\zeta(\omega)e_{k_0}(\theta)\bigr)\),
so \(\lambda^\omega\) is a probability measure.  It is not Haar, since
\[
        \widehat{\lambda^\omega}(k_0)
        =
        \varepsilon\overline{\zeta(\omega)},
        \qquad
        \widehat{\lambda^\omega}(-k_0)
        =
        \varepsilon\zeta(\omega).
\]
The zero Fourier mode is stationary.  For \(-k_0\),
\[
\begin{aligned}
        \widehat{\mu^\omega*\lambda^\omega}(-k_0)
        &=
        \widehat{\mu^\omega}(-k_0)
        \widehat{\lambda^\omega}(-k_0)  \\
        &=
        \overline{a(\omega)}\,\varepsilon\zeta(\omega)
        =
        \varepsilon\zeta(\Theta\omega)
        =
        \widehat{\lambda^{\Theta\omega}}(-k_0).
\end{aligned}
\]
For \(k_0\), the conjugate identity gives
\(\widehat{\mu^\omega*\lambda^\omega}(k_0) = \widehat{\lambda^{\Theta\omega}}(k_0)\).
All other nonzero Fourier modes vanish on both sides.  Fourier--Stieltjes
uniqueness gives
\(\mu^\omega*\lambda^\omega=\lambda^{\Theta\omega}, \ \PP\text{-a.e.}\)
Thus \(\lambda^\bullet\) is a non-Haar stationary torus family, contradicting \((5)\).
Therefore \((4)\) holds.

The theorem is proved.
\end{proof}

\subsection{Linear cocycle and top Lyapunov exponent}

The linear skew product over the mixed base is
\[
        \mathscr F:
        \Omega\times G^{\NN}\times\TT^d\times\RR^r
        \longrightarrow
        \Omega\times G^{\NN}\times\TT^d\times\RR^r
\]
defined by
\[
        \mathscr F(\omega,\xi,\theta,v)
        :=
        \bigl(
        \Theta\omega,
        \sigma\xi,
        \theta+\mathfrak a(\xi_0),
        \mathcal A(\xi_0)(\theta)v
        \bigr).
\]
Thus \(\mathscr F\) covers \(\mathfrak F\), and its vector-fiber action over
\((\omega,\xi,\theta)\) is \(\mathcal A(\xi_0)(\theta)\).

Using the path-sum notation above, set
\(\mathcal A_0(\xi,\theta):=I\),
and, for \(n\ge1\),
\[
        \mathcal A_n(\xi,\theta)
        :=
        \mathcal A(\xi_{n-1})\bigl(\theta+\Gamma_{n-1}(\xi)\bigr)
        \cdots
        \mathcal A(\xi_1)\bigl(\theta+\Gamma_1(\xi)\bigr)
        \mathcal A(\xi_0)(\theta).
\]
We call \(\mathcal A_n(\xi,\theta)\) the \emph{\(n\)-step matrix product}.
Then
\[
        \mathscr F^n(\omega,\xi,\theta,v)
        =
        \bigl(
        \Theta^n\omega,
        \sigma^n\xi,
        \theta+\Gamma_n(\xi),
        \mathcal A_n(\xi,\theta)v
        \bigr).
\]
Although \(\mathcal A_n(\xi,\theta)\) is not written with an explicit
\(\omega\)-dependence, its law depends on \(\omega\) through
\[
        \mathbf P^\omega
        =
        \bigotimes_{n\ge0}\nu^{\Theta^n\omega}.
\]

For Lyapunov-exponent statements, we assume the logarithmic integrability
condition
\begin{equation}\label{eq:log-integrability-G}
        \int_\Omega\int_G
        \Lambda_{\mathcal A}(g)
        \,d\nu^\omega(g)\,d\PP(\omega)
        <\infty,
\end{equation}
where
\(\Lambda_{\mathcal A}(g) := \sup_{\theta\in\TT^d} \log^+\|\mathcal A(g)(\theta)\|\).
Since the matrices lie in \(\SL_r(\RR)\), this also gives inverse logarithmic
integrability.  Indeed, for \(B\in\SL_r(\RR)\),
\[
        \|B^{-1}\|\le \|B\|^{r-1},
        \qquad
        \log^+\|B^{-1}\|
        \le
        (r-1)\log^+\|B\|.
\]
Consequently,
\[
        \sup_{\theta\in\TT^d}
        \log^+\|\mathcal A(g)(\theta)^{-1}\|
        \le
        (r-1)\Lambda_{\mathcal A}(g).
\]
We shall use this inverse estimate repeatedly.  For notational convenience set
\(
        C_{\mathrm{inv}}:=\max\{1,r-1\}.
\)

For \(n\ge1\), define
\(f_n(\omega,\xi,\theta) := \log\|\mathcal A_n(\xi,\theta)\|\).
The cocycle identity
\[
        \mathcal A_{n+m}(\xi,\theta)
        =
        \mathcal A_m(\sigma^n\xi,\theta+\Gamma_n(\xi))
        \mathcal A_n(\xi,\theta)
\]
implies
\[
        f_{n+m}(\omega,\xi,\theta)
        \le
        f_n(\omega,\xi,\theta)
        +
        f_m(\mathfrak F^n(\omega,\xi,\theta)).
\]
Moreover, \eqref{eq:log-integrability-G} gives
\(f_1^+\in L^1(\rho)\).
Kingman's subadditive ergodic theorem therefore applies \cite{Kin68}.  Hence
\[
        \lim_{n\to\infty}
        \frac1n\log\|\mathcal A_n(\xi,\theta)\|
\]
exists for \(\rho\)-a.e. \((\omega,\xi,\theta)\).  If
\((\Omega\times G^{\NN}\times\TT^d,\mathfrak F,\rho)\)
is ergodic, the limit is constant \(\rho\)-a.e.. This constant is the top
Lyapunov exponent of \(\mathscr F\), denoted by \(L_1(\rho)\).  Thus
\[
        L_1(\rho)
        =
        \lim_{n\to\infty}
        \frac1n\log\|\mathcal A_n(\xi,\theta)\|
        \qquad
        \text{for }\rho\text{-a.e. }(\omega,\xi,\theta).
\]
In the same ergodic case, Kingman's theorem also gives
\[
        L_1(\rho)
        =
        \lim_{n\to\infty}
        \int_{\Omega\times G^{\NN}\times\TT^d}
        \frac1n\log\|\mathcal A_n(\xi,\theta)\|\,
        d\rho(\omega,\xi,\theta)
\]
and
\[
        L_1(\rho)
        =
        \inf_{n\ge1}
        \frac1n
        \int_{\Omega\times G^{\NN}\times\TT^d}
        \log\|\mathcal A_n(\xi,\theta)\|\,
        d\rho(\omega,\xi,\theta).
\]

\begin{remark}[Norm convention]
\label{rem:norm-convention}
Throughout the paper, \(\|\cdot\|\) denotes the fixed Euclidean norm on
\(\RR^r\) and the induced operator norm.  Replacing it by any other norm changes
\(\log\|B_n\cdots B_1\|\) by a uniformly bounded additive error.  Indeed, if
\(C^{-1}\|v\|\le \|v\|_*\le C\|v\|, \ v\in\RR^r\),
then the induced operator norms satisfy
\(C^{-2}\|M\|\le \|M\|_*\le C^2\|M\|, \ M\in\mathrm{End}(\RR^r)\).
Therefore
\[
        \left|
        \log\|B_n\cdots B_1\|_*-
        \log\|B_n\cdots B_1\|
        \right|
        \le 2\log C,
\]
and division by \(n\) leaves the top Lyapunov exponent unchanged.
\end{remark}

\section{Stationary Projective Families and the Environment-Wise Furstenberg Formula}
\label{sec:stationary-families}

This section supplies the projective variational layer for the mixed base fixed
in Section~\ref{sec:model}.  The projective action turns the matrix cocycle
into a random transformation of
\(Q=\TT^d\times\mathbb P(\RR^r)\), and the moving one-step cocycle law
\(\nu^\omega\) leads naturally to stationary projective families rather than
to one stationary projective measure.  We construct the compact stationary
Young-measure domain, record exact equivariance for later use in
Section~\ref{sec:exact-equivariance}, use mixed-base ergodicity to force Haar
torus marginals, and then prove the environment-wise Furstenberg formula.

\subsection{Stationary projective families and exact equivariance}

We first pass from the linear cocycle to its projective dynamics.  This provides
the variational domain for the Furstenberg functional, while the exact version
of equivariance is recorded for the zero-exponent rigidity argument in the next
section.

The \emph{projective state space} is
\(Q:=\TT^d\times\mathbb P(\RR^r)\).
For \(g=(\alpha,A)\in G\), define the projective action on \(Q\) by
\(
        g\cdot(\theta,\hat p)
        :=
        \bigl(\theta+\alpha,\widehat{A(\theta)p}\bigr),
\)
where \(p\in\hat p\) is any nonzero representative.  This is independent of the
choice of representative.

\begin{definition}[Stationary projective family]
A measurable family
\(\eta^\bullet=\{\eta^\omega\}_{\omega\in\Omega}, \  \eta^\omega\in\Prob(Q)\),
is called a \emph{stationary projective family} if
\begin{equation}\label{eq:averaged-stationary}
        \eta^{\Theta\omega}
        =
        \int_G g_*\eta^\omega\,d\nu^\omega(g)
        \qquad
        \text{for }\PP\text{-a.e. }\omega .
\end{equation}
\end{definition}

We write
\(\mathscr S_Q := \{\eta^\bullet:\eta^\bullet\text{ is a stationary projective family}\}\).

\begin{definition}[Exact equivariant projective family]
\label{def:exactly-equivariant-family}
A measurable family \(\eta^\bullet=\{\eta^\omega\}_{\omega\in\Omega}\), with
\(\eta^\omega\in\Prob(Q)\), is called an \emph{exact equivariant projective family
on \(Q\)} if there is a Borel set \(\Omega_0\subset\Omega\) with
\(\PP(\Omega_0)=1\) such that, for every \(\omega\in\Omega_0\) and every
\(g\in\supp\nu^\omega\),
\begin{equation}\label{eq:exact-equivariant}
        g_*\eta^\omega=\eta^{\Theta\omega}.
\end{equation}
Here \(\supp\nu^\omega\) is the topological support of \(\nu^\omega\) in \(G\).
\end{definition}

\begin{remark}[Averaged stationarity versus exact equivariance]
Averaged stationarity is the fixed-point relation obtained after integration
against \(\nu^\omega\).  Exact equivariance is the stronger exact support
transport identity required for every support element.
It is not needed for the Furstenberg formula itself, but it is the rigidity
output used in Section~\ref{sec:exact-equivariance}.
\end{remark}

To prove that the stationary domain is nonempty and compact, we use the
Young-measure compactification of random probability fields over \(Q\), with
the stable topology of \cite[Chapter~2, Section~2.1]{CRV04}.  Let
\[
\mathscr Y_{\PP}(Q)
:=
\bigl\{\lambda\in\Prob(\Omega\times Q):(\pi_\Omega)_*\lambda=\PP\bigr\},
\]
where \(\pi_\Omega:\Omega\times Q\to\Omega\) is the canonical projection.  We
endow \(\mathscr Y_{\PP}(Q)\) with the \emph{stable topology}, namely the coarsest
topology for which
\[
        \lambda\longmapsto
        \int_{\Omega\times Q}\varphi(\omega,x)\,d\lambda(\omega,x)
\]
is continuous for every \(\varphi\in L^1(\Omega;C(Q))\).  Equivalently,
\(\varphi\) is jointly measurable, \(\varphi(\omega,\cdot)\in C(Q)\) for
\(\PP\)-a.e. \(\omega\), and
\[
        \int_\Omega \sup_{x\in Q}|\varphi(\omega,x)|\,d\PP(\omega)<\infty .
\]

Since \((\Omega,\cF,\PP)\) and \(Q\) are standard Borel, the disintegration
theorem \cite[Theorem~8.5]{Kallenberg21} gives, for every
\(\lambda\in\mathscr Y_{\PP}(Q)\), a measurable probability kernel
\(\omega\mapsto\lambda_\omega\) such that
\(\lambda(d\omega,dx)=\PP(d\omega)\lambda_\omega(dx)\).  Thus
\(\mathscr Y_{\PP}(Q)\) is the Young-measure space over \((\Omega,\PP)\) with
values in the compact metric space \(Q\).  Because \(Q\) is compact, this class
is strictly tight.  The stable compactness theorem
\cite[Theorem~4.3.5(A1)]{CRV04} gives relative compactness, and the
fixed-marginal condition is stable closed; hence
\(\mathscr Y_{\PP}(Q)\) is compact in the stable topology.  It is also convex and
nonempty.

\begin{proposition}[Existence via Young measures]
\label{prop:stationary-existence}
There exists at least one stationary projective family
\(
\eta^\bullet=\{\eta^\omega\}_{\omega\in\Omega}
\)
on \(Q\).
\end{proposition}

\begin{proof}
Fix \(\lambda\in\mathscr Y_{\PP}(Q)\), and write
\(\lambda(d\omega,dx)=\PP(d\omega)\lambda_\omega(dx)\).  Define an affine operator
\(\mathcal P:\mathscr Y_{\PP}(Q)\to\mathscr Y_{\PP}(Q)\) by duality: for every
\(\varphi\in L^1(\Omega;C(Q))\),
\begin{equation}\label{eq:def-P-on-young}
\int_{\Omega\times Q}\varphi(\omega,x)\,d(\mathcal P\lambda)(\omega,x)
:=
\int_\Omega\int_Q
\left[
\int_G
\varphi(\Theta\omega,g\cdot x)\,d\nu^\omega(g)
\right]
d\lambda_\omega(x)\,d\PP(\omega).
\end{equation}

For \(\varphi\in L^1(\Omega;C(Q))\), set
\[
        (\mathcal L\varphi)(\omega,x)
        :=
        \int_G \varphi(\Theta\omega,g\cdot x)\,d\nu^\omega(g).
\]
The continuity of the action \(G\times Q\to Q\), the measurability of the kernel
\(\omega\mapsto\nu^\omega\), and dominated convergence imply that
\(\mathcal L\varphi\) is jointly measurable and that
\((\mathcal L\varphi)(\omega,\cdot)\in C(Q)\) for \(\PP\)-a.e. \(\omega\).  Moreover,
\[
        \| (\mathcal L\varphi)(\omega,\cdot)\|_{C(Q)}
        \le
        \|\varphi(\Theta\omega,\cdot)\|_{C(Q)}.
\]
Using \(\Theta_*\PP=\PP\), we get
\[
        \int_\Omega
        \| (\mathcal L\varphi)(\omega,\cdot)\|_{C(Q)}\,d\PP(\omega)
        \le
        \int_\Omega
        \|\varphi(\omega,\cdot)\|_{C(Q)}\,d\PP(\omega)<\infty.
\]
Thus \(\mathcal L\) maps \(L^1(\Omega;C(Q))\) into itself.

The operator \(\mathcal P\) preserves the base marginal.  Indeed, for bounded
measurable \(\chi:\Omega\to\RR\), viewed as independent of \(x\),
\[
\begin{aligned}
\int_{\Omega\times Q}\chi(\omega)\,d(\mathcal P\lambda)(\omega,x)
&=
\int_\Omega \chi(\Theta\omega)\,d\PP(\omega)  \\
&=
\int_\Omega \chi(\omega)\,d\PP(\omega).
\end{aligned}
\]
Hence \(\mathcal P(\mathscr Y_{\PP}(Q))\subset\mathscr Y_{\PP}(Q)\).  Since
\(\mathcal L\varphi\in L^1(\Omega;C(Q))\), the dual formula
\eqref{eq:def-P-on-young} also shows that \(\mathcal P\) is continuous in the
stable topology.

The stable topology is induced by the integration pairings above, so
\(\mathscr Y_{\PP}(Q)\) is a compact convex subset of a Hausdorff locally convex
space.  The Tychonoff fixed point theorem \cite{Tychonoff35} therefore gives
\(\lambda\in\mathscr Y_{\PP}(Q)\) with \(\mathcal P\lambda=\lambda\).  Disintegrate
it as
\(\lambda(d\omega,dx)=\PP(d\omega)\eta^\omega(dx)\).
It remains only to rewrite the fixed-point identity pointwise in the moving
environment.  Since \(\Theta\) is invertible, define
\[
        \zeta^\omega
        :=
        \int_G g_*\eta^{\Theta^{-1}\omega}\,
        d\nu^{\Theta^{-1}\omega}(g)
        \qquad
        \text{for }\PP\text{-a.e. }\omega .
\]
For every bounded measurable \(\chi:\Omega\to\RR\) and every \(f\in C(Q)\), the
change of variables \(\omega\mapsto\Theta\omega\), the definition of
\(\mathcal P\), and \(\mathcal P\lambda=\lambda\) give
\[
\begin{aligned}
\int_\Omega \chi(\omega)\int_Q f\,d\zeta^\omega\,d\PP(\omega)
&=
\int_\Omega \chi(\Theta\omega)
\int_Q\int_G f(g\cdot x)\,d\nu^\omega(g)\,d\eta^\omega(x)\,d\PP(\omega) \\
&=
\int_{\Omega\times Q}\chi(\omega)f(x)\,
d(\mathcal P\lambda)(\omega,x) \\
&=
\int_\Omega \chi(\omega)\int_Q f\,d\eta^\omega\,d\PP(\omega).
\end{aligned}
\]
Uniqueness in the disintegration theorem
\cite[Theorem~8.5]{Kallenberg21}, applied over the fixed base marginal \(\PP\),
yields
\[
        \eta^\omega=
        \int_G g_*\eta^{\Theta^{-1}\omega}\,
        d\nu^{\Theta^{-1}\omega}(g)
        \qquad
        \text{for }\PP\text{-a.e. }\omega .
\]
Replacing \(\omega\) by \(\Theta\omega\) gives \eqref{eq:averaged-stationary}.
Thus \(\eta^\bullet\in\mathscr S_Q\).
\end{proof}

\subsection{Haar torus marginals and projective disintegration}

The preceding subsection gives stationary projective families without using
mixed-base ergodicity.  We now use the ergodicity criterion from
Theorem~\ref{thm:ergodicity-random-environment} to identify their torus
marginals.  This is the step that lets later arguments use the fixed Haar torus
marginal rather than an arbitrary stationary torus family.

\begin{theorem}[Haar torus marginal]
\label{prop:Haar-marginal}
Assume that the measure-preserving system
\((\Omega\times G^{\NN}\times\TT^d,\mathfrak F,\rho)\)
is ergodic.  If
\(\eta^\bullet=\{\eta^\omega\}_{\omega\in\Omega}\) is a stationary projective
family on \(Q\), then \(\eta^\omega\) has Haar torus marginal \(m\) for
\(\PP\)-a.e. \(\omega\).
\end{theorem}

\begin{proof}
Let
\(\pi_{\Sigma}:Q=\TT^d\times \mathbb P(\RR^r)\to\TT^d, \ \pi_{\Sigma}(\theta,\hat p):=\theta\),
and set \(\lambda^\omega:=(\pi_{\Sigma})_*\eta^\omega\).  The family
\(\lambda^\bullet\) is measurable.  For \(\varphi\in C(\TT^d,\mathbb C)\),
stationarity of \(\eta^\bullet\) and
\(\pi_{\Sigma}(g\cdot x)=\pi_{\Sigma}(x)+\mathfrak a(g)\) give, for
\(\PP\)-a.e. \(\omega\),
\[
\begin{aligned}
\int_{\TT^d}\varphi\,d\lambda^{\Theta\omega}
&=
\int_G\int_Q
\varphi(\pi_{\Sigma}(g\cdot x))\,d\eta^\omega(x)\,d\nu^\omega(g) \\
&=
\int_{\TT^d}\int_{\TT^d}
\varphi(\theta+\alpha)\,d\mu^\omega(\alpha)\,d\lambda^\omega(\theta) \\
&=
\int_{\TT^d}\varphi\,d(\mu^\omega*\lambda^\omega).
\end{aligned}
\]
Hence
\(\lambda^{\Theta\omega}=\mu^\omega*\lambda^\omega \qquad \text{for }\PP\text{-a.e. }\omega\).
Thus \(\lambda^\bullet\) is a stationary torus family.  By item \textup{(5)} of
Theorem~\ref{thm:ergodicity-random-environment}, mixed-base ergodicity makes the Haar
torus family the only stationary torus family.  Therefore \(\lambda^\omega=m\) for
\(\PP\)-a.e. \(\omega\).
\end{proof}

Under the hypotheses of Theorem~\ref{prop:Haar-marginal}, we use the
following \emph{Haar projective disintegration}.  For every stationary projective family
\(\eta^\bullet\in\mathscr S_Q\), choose versions such that
\[
        \eta^\omega(d\theta,d\hat p)
        =
        m(d\theta)\,\eta^\omega_\theta(d\hat p)
        \qquad
        \text{for }\PP\text{-a.e. }\omega,
\]
where \((\omega,\theta)\mapsto\eta^\omega_\theta\in\Prob(\mathbb P(\RR^r))\) is
jointly measurable.  Existence and uniqueness of this disintegration follow
from \cite[Theorem~8.5]{Kallenberg21}; the conditional kernel is unique up to a
\((\PP\otimes m)\)-null set.  For \(B\in\SL_r(\RR)\), the notation
\(B_*\eta^\omega_\theta\) denotes the push-forward of \(\eta^\omega_\theta\) under
the projective action of \(B\) on \(\mathbb P(\RR^r)\).

With this convention, averaged stationarity becomes the projective-fiber equation
\begin{equation}\label{eq:stationary-disintegration}
        \eta^{\Theta\omega}_{\theta}
        =
        \int_G
        \mathcal A(g)(\theta-\mathfrak a(g))_*
        \eta^\omega_{\theta-\mathfrak a(g)}
        \,d\nu^\omega(g)
\end{equation}
for \(\PP\)-a.e. \(\omega\) and \(m\)-a.e. \(\theta\).  Indeed, for every
\(\Phi\in C(Q)\), averaged stationarity and the Haar disintegration give
\[
\begin{aligned}
\int_Q \Phi\,d\eta^{\Theta\omega}
&=
\int_G\int_{\TT^d}\int_{\mathbb P(\RR^r)}
\Phi\bigl(\theta+\mathfrak a(g),
\mathcal A(g)(\theta)\hat p\bigr)
\,d\eta^\omega_\theta(\hat p)\,dm(\theta)\,d\nu^\omega(g) \\
&=
\int_{\TT^d}\int_G\int_{\mathbb P(\RR^r)}
\Phi\bigl(\theta,\mathcal A(g)(\theta-\mathfrak a(g))\hat p\bigr)
\,d\eta^\omega_{\theta-\mathfrak a(g)}(\hat p)\,
d\nu^\omega(g)\,dm(\theta),
\end{aligned}
\]
where the second equality uses translation invariance of \(m\).  Uniqueness of
Rokhlin disintegration over \(m\) yields
\eqref{eq:stationary-disintegration}.

If \(\eta^\bullet\) is exact equivariant and the mixed base is ergodic, then the same
Haar disintegration gives the projective-fiber form of exact equivariance.  After choosing the
versions above on a common full-measure set of environments, there exists a
Borel set \(\Omega_0\subset\Omega\) with \(\PP(\Omega_0)=1\) such that, for every
\(\omega\in\Omega_0\) and every \(g\in\supp\nu^\omega\),
\begin{equation}\label{eq:exact-equivariant-disintegration}
        \mathcal A(g)(\theta)_*\eta^\omega_\theta
        =
        \eta^{\Theta\omega}_{\theta+\mathfrak a(g)}
\end{equation}
for \(m\)-a.e. \(\theta\).  The exceptional \(m\)-null set may depend on
\((\omega,g)\).

\subsection{Environment-Wise Furstenberg Formula}

We now define the one-step projective growth functional on \(\mathscr S_Q\) and
compare its maximum with the annealed operator-norm growth. 

For \(B\in\SL_r(\RR)\) and \(\hat p\in\mathbb P(\RR^r)\), define the
\emph{projective growth observable}
\begin{equation}\label{eq:euclidean-projective-growth}
        \sigma(B,\hat p)
        :=
        \log\frac{\|Bp\|}{\|p\|},
        \qquad p\in\hat p,
        \quad p\neq0 .
\end{equation}
This is independent of the representative \(p\). If \(\|p\|=1\), then
\(\sigma(B,\hat p)=\log\|Bp\|\).

For \(g=(\alpha,A)\in G\) and \(x=(\theta,\hat p)\in Q\), set
\(\Psi(g,x):=\sigma(A(\theta),\hat p)\)
and
\[
        \psi(\omega,x):=\int_G \Psi(g,x)\,d\nu^\omega(g).
\]

\begin{lemma}[Integrability of the averaged projective observable]
\label{lem:psi-L1-CQ}
Under the logarithmic integrability condition~\eqref{eq:log-integrability-G},
\(\psi\in L^1(\Omega;C(Q))\).
\end{lemma}

\begin{proof}
The inverse estimate following \eqref{eq:log-integrability-G} gives the uniform
bound
\begin{equation}\label{eq:Psi-two-sided-bound}
        |\Psi(g,x)|
        \le
        C_{\mathrm{inv}}\bigl(1+\Lambda_{\mathcal A}(g)\bigr)
        \qquad
        (g\in G,
        \ x\in Q).
\end{equation}
Indeed, if \(g=(\alpha,A)\), \(x=(\theta,\hat p)\), and \(p\in\hat p\) is chosen
with \(\|p\|=1\), then
\[
        \log\|A(\theta)p\|
        \le
        \log^+\|A(\theta)\|
        \le
        \Lambda_{\mathcal A}(g),
\]
while
\[
        -\log\|A(\theta)p\|
        \le
        \log^+\|A(\theta)^{-1}\|
        \le
        C_{\mathrm{inv}}\Lambda_{\mathcal A}(g).
\]
Thus \eqref{eq:Psi-two-sided-bound} is integrable with respect to
\(\PP(d\omega)\nu^\omega(dg)\).

For \(\PP\)-a.e. \(\omega\), dominated convergence gives
\(\psi(\omega,\cdot)\in C(Q)\).  The measurability of \((g,x)\mapsto\Psi(g,x)\)
and of \(\omega\mapsto\nu^\omega\) gives joint measurability of
\((\omega,x)\mapsto\psi(\omega,x)\).  Since \(Q\) is compact metric and \(C(Q)\)
is separable, this is Bochner measurability of
\(\omega\mapsto\psi(\omega,\cdot)\).  Finally,
\[
        \|\psi(\omega,\cdot)\|_{C(Q)}
        \le
        \int_G C_{\mathrm{inv}}\bigl(1+\Lambda_{\mathcal A}(g)\bigr)
        \,d\nu^\omega(g),
\]
and the right-hand side is integrable in \(\omega\).
\end{proof}

For \(\eta^\bullet\in\mathscr S_Q\), define the Furstenberg functional
\begin{equation}\label{eq:alpha-env}
\alpha(\eta^\bullet)
:=
\int_\Omega\int_Q
\psi(\omega,x)\,d\eta^\omega(x)\,d\PP(\omega).
\end{equation}
Identifying \(\eta^\bullet\) with the Young measure
\(\lambda_\eta(d\omega,dx)=d\PP(\omega)d\eta^\omega(x)\), the set
\(\mathscr S_Q\) is the fixed-point set of the continuous affine map
\(\mathcal P\) on \(\mathscr Y_{\PP}(Q)\).  Hence \(\mathscr S_Q\) is compact in
the stable topology.  It is nonempty by Proposition~\ref{prop:stationary-existence}
and convex by linearity of \eqref{eq:averaged-stationary}.  Since
\(\psi\in L^1(\Omega;C(Q))\), the functional \(\alpha\) is continuous and affine
on \(\mathscr S_Q\).  Therefore it attains its maximum:
\[
        \beta:=\max_{\eta^\bullet\in\mathscr S_Q}\alpha(\eta^\bullet).
\]

The next two propositions give the annealed Furstenberg--Kifer upper bound.  The
first is a compactness statement for empirical projective laws tested against an
arbitrary stable observable.  Applying it to \(\psi\) gives the fixed-direction
bound, and the projective averaging lemma then upgrades fixed directions to the
operator norm.

\begin{proposition}[Relative Furstenberg--Kifer compactness bound]
\label{lem:annealed-relative-FK-upper}
Let \(\varphi\in L^1(\Omega;C(Q))\) be real-valued.  Fix
\(\hat v\in\mathbb P(\RR^r)\), and define
\(X_0=(\theta,\hat v), \ X_{j+1}:=\xi_j\cdot X_j\).
Let \(\lambda_j\in\Prob(\Omega\times Q)\) be the law of
\((\Theta^j\omega,X_j)\) under
\(d\rho(\omega,\xi,\theta) = d\PP(\omega)\,d\mathbf P^\omega(\xi)\,dm(\theta)\),
and set
\(
        \bar\lambda_n:=\frac1n\sum_{j=0}^{n-1}\lambda_j .
\)
Then
\[
\limsup_{n\to\infty}
\int_{\Omega\times Q}\varphi\,d\bar\lambda_n
\le
\sup_{\eta^\bullet\in\mathscr S_Q}
\int_\Omega\int_Q
\varphi(\omega,x)\,d\eta^\omega(x)\,d\PP(\omega).
\]
\end{proposition}

\begin{proof}
Each \(\lambda_j\) has \(\Omega\)-marginal \(\PP\): for bounded measurable
\(a:\Omega\to\RR\),
\[
\int_{\Omega\times Q}a(\omega)\,d\lambda_j(\omega,x)
=
\int_\Omega a(\Theta^j\omega)\,d\PP(\omega)
=
\int_\Omega a(\omega)\,d\PP(\omega).
\]
Thus \(\lambda_j,\bar\lambda_n\in\mathscr Y_{\PP}(Q)\).  Choose
\(n_k\to\infty\) along the limsup.  By compactness of \(\mathscr Y_{\PP}(Q)\),
passing to a subsequence gives
\(\bar\lambda_{n_k}\to\lambda\) stably for some
\(\lambda\in\mathscr Y_{\PP}(Q)\).  Since \(\varphi\in L^1(\Omega;C(Q))\),
\[
\int_{\Omega\times Q}\varphi\,d\lambda
=
\lim_{k\to\infty}\int_{\Omega\times Q}\varphi\,d\bar\lambda_{n_k}
=
\limsup_{n\to\infty}\int_{\Omega\times Q}\varphi\,d\bar\lambda_n .
\]

It remains to show that \(\lambda\) is stationary.  For
\(\phi\in L^1(\Omega;C(Q))\), use the same operator
\[
        (\mathcal L\phi)(\omega,x)
        :=
        \int_G \phi(\Theta\omega,g\cdot x)\,d\nu^\omega(g).
\]
Then \(\mathcal L\phi\in L^1(\Omega;C(Q))\).  If
\(\mathcal F_j:=\sigma(\omega,\theta,\xi_0,\dots,\xi_{j-1})\), then \(X_j\) is
\(\mathcal F_j\)-measurable and, conditionally on \(\mathcal F_j\), the law of
\(\xi_j\) is \(\nu^{\Theta^j\omega}\).  Hence
\[
        \int_{\Omega\times Q}\phi\,d\lambda_{j+1}
        =
        \int_{\Omega\times Q}\mathcal L\phi\,d\lambda_j .
\]
Averaging in \(j\) yields
\[
\int_{\Omega\times Q}(\phi-\mathcal L\phi)\,d\bar\lambda_n
=
\frac1n
\left(
\int_{\Omega\times Q}\phi\,d\lambda_0
-
\int_{\Omega\times Q}\phi\,d\lambda_n
\right).
\]
The right-hand side tends to zero because all \(\lambda_j\) have
\(\Omega\)-marginal \(\PP\) and
\[
\left|\int_{\Omega\times Q}\phi\,d\lambda_j\right|
\le
\int_\Omega\|\phi(\omega,\cdot)\|_{C(Q)}\,d\PP(\omega).
\]
Passing to the stable limit gives
\[
        \int_{\Omega\times Q}\phi\,d\lambda
        =
        \int_{\Omega\times Q}\mathcal L\phi\,d\lambda
        \qquad
        \forall\,\phi\in L^1(\Omega;C(Q)).
\]
Writing \(\lambda(d\omega,dx)=d\PP(\omega)d\eta^\omega(x)\), this is exactly
\eqref{eq:averaged-stationary}.  Thus \(\eta^\bullet\in\mathscr S_Q\), and the
limsup is bounded by the stated supremum.
\end{proof}

\begin{proposition}[Fixed-direction annealed upper bound]
\label{lem:annealed-env-FK-upper}
Assume that \eqref{eq:log-integrability-G} holds.  Let \(v\in\RR^r\) be a fixed
Euclidean unit vector.  Then
\[
\limsup_{n\to\infty}
\int_{\Omega\times G^{\NN}\times\TT^d}
\frac1n\log\|\mathcal A_n(\xi,\theta)v\|\,
 d\rho(\omega,\xi,\theta)
\le
\beta.
\]
\end{proposition}

\begin{proof}
Let \(\hat v\) be the projective class of \(v\), and let
\(X_j,\lambda_j,\bar\lambda_n\) be the process and empirical laws from
Proposition~\ref{lem:annealed-relative-FK-upper}.  By Lemma~\ref{lem:psi-L1-CQ},
\(\psi\in L^1(\Omega;C(Q))\).

With \(\mathcal F_j:=\sigma(\omega,\theta,\xi_0,\dots,\xi_{j-1})\), the conditional
law of \(\xi_j\) given \(\mathcal F_j\) is \(\nu^{\Theta^j\omega}\).  Therefore
\[
\EE_\rho[\Psi(\xi_j,X_j)\mid\mathcal F_j]
=
\int_G\Psi(g,X_j)\,d\nu^{\Theta^j\omega}(g)
=
\psi(\Theta^j\omega,X_j).
\]
Writing \(X_j=(\theta_j,\hat p_j)\), choose Euclidean unit representatives
\(p_0=v\) and
\[
        p_{j+1}
        =
        \frac{\mathcal A(\xi_j)(\theta_j)p_j}
        {\|\mathcal A(\xi_j)(\theta_j)p_j\|}.
\]
The projective telescoping identity gives
\[
        \sum_{j=0}^{n-1}\Psi(\xi_j,X_j)
        =
        \log\|\mathcal A_n(\xi,\theta)v\|.
\]
Consequently,
\[
\int
\frac1n\log\|\mathcal A_n(\xi,\theta)v\|\,d\rho
=
\frac1n\sum_{j=0}^{n-1}
\int\psi(\Theta^j\omega,X_j)\,d\rho
=
\int_{\Omega\times Q}\psi\,d\bar\lambda_n.
\]
Applying Proposition~\ref{lem:annealed-relative-FK-upper} with
\(\varphi=\psi\) gives the claim.
\end{proof}

Let \(S^{r-1}=\{v\in\RR^r:\|v\|=1\}\).  The reference measure
\(m_{\mathbb P}\) on \(\mathbb P(\RR^r)\) is the push-forward of normalized
rotation-invariant probability measure on \(S^{r-1}\) under the antipodal
quotient map \(S^{r-1}\to\mathbb P(\RR^r)\).

\begin{lemma}[Projective averaging estimate]
\label{lem:projective-averaging}
There exists a constant \(C_r<\infty\) such that, for every \(B\in\SL_r(\RR)\),
\[
        \log\|B\|
        \le
        C_r
        +
        \int_{\mathbb P(\RR^r)}
        \sigma(B,\hat p)\,dm_{\mathbb P}(\hat p).
\]
\end{lemma}

\begin{proof}
Let \(s_1=\|B\|\), and let \(u\) be a Euclidean unit right singular vector
corresponding to \(s_1\).  For a Euclidean unit representative \(p\in\hat p\),
\(
        \|Bp\|\ge \|B\|\,|\langle p,u\rangle| .
\)
Taking logarithms and integrating with respect to \(m_{\mathbb P}\) gives
\begin{equation}\label{eq:euclidean-projective-averaging}
        \log\|B\|
        \le
        C_r
        +
        \int_{\mathbb P(\RR^r)}
        \log\|Bp\|\,dm_{\mathbb P}(\hat p),
\end{equation}
where
\[
        C_r:=-\int_{\mathbb P(\RR^r)}
        \log|\langle p,e_1\rangle|\,dm_{\mathbb P}(\hat p)<\infty .
\]
For \(r=1\) this is immediate.  For \(r\ge2\), finiteness follows from the
spherical strip estimate
\[
        \sigma_{S^{r-1}}
        \{p\in S^{r-1}:|\langle p,e_1\rangle|\le t\}
        \le c_r t,
        \qquad 0<t<1,
\]
and the layer-cake representation.  Since \(\log\|Bp\|=\sigma(B,\hat p)\),
\eqref{eq:euclidean-projective-averaging} proves the lemma.
\end{proof}

\begin{proposition}[Annealed operator-norm upper bound]
\label{prop:annealed-operator-FK-upper}
Assume that the logarithmic integrability condition
\eqref{eq:log-integrability-G} holds.  Then
\[
\limsup_{n\to\infty}
\int_{\Omega\times G^{\NN}\times\TT^d}
\frac1n\log\|\mathcal A_n(\xi,\theta)\|\,
 d\rho(\omega,\xi,\theta)
\le
\beta.
\]
\end{proposition}

\begin{proof}
For \(\hat p\in\mathbb P(\RR^r)\), let \(p\in\hat p\) be a Euclidean unit
representative, and set
\[
        a_n(\hat p)
        :=
        \int_{\Omega\times G^{\NN}\times\TT^d}
        \frac1n\log\|\mathcal A_n(\xi,\theta)p\|\,
        d\rho(\omega,\xi,\theta).
\]
The standing logarithmic control gives, for every Euclidean unit \(p\),
\begin{equation}\label{eq:positive-bound-An}
        \log^+\|\mathcal A_n(\xi,\theta)p\|
        \le
        \log^+\|\mathcal A_n(\xi,\theta)\|
        \le
        \sum_{j=0}^{n-1}\Lambda_{\mathcal A}(\xi_j).
\end{equation}
The inverse estimate following \eqref{eq:log-integrability-G} also gives
\[
        -\log\|\mathcal A_n(\xi,\theta)p\|
        \le
        \log^+\|\mathcal A_n(\xi,\theta)^{-1}\|
        \le
        C_{\mathrm{inv}}
        \sum_{j=0}^{n-1}\Lambda_{\mathcal A}(\xi_j).
\]
The right-hand side is integrable with respect to \(\rho\), uniformly in
\(\hat p\), because each \(\xi_j\) has mean logarithmic control
\[
        M_{\mathcal A}
        :=
        \int_\Omega\int_G
        \Lambda_{\mathcal A}(g)\,d\nu^\omega(g)\,d\PP(\omega)<\infty
\]
by \(\Theta\)-invariance of \(\PP\).  Hence Fubini's theorem applies with respect
to \(\rho\otimes m_{\mathbb P}\).

Applying Lemma~\ref{lem:projective-averaging} to
\(B=\mathcal A_n(\xi,\theta)\), dividing by \(n\), and integrating yields
\[
\int
\frac1n\log\|\mathcal A_n(\xi,\theta)\|\,d\rho
\le
\frac{C_r}{n}
+
\int_{\mathbb P(\RR^r)}a_n(\hat p)\,dm_{\mathbb P}(\hat p).
\]
Moreover, \eqref{eq:positive-bound-An} gives the uniform upper bound
\[
        a_n(\hat p)\le M_{\mathcal A}
        \qquad
        \text{for all }\hat p\in\mathbb P(\RR^r)\text{ and }n\ge1.
\]
Fatou's lemma applied to \(M_{\mathcal A}-a_n(\hat p)\ge0\) gives
\[
\limsup_{n\to\infty}
\int_{\mathbb P(\RR^r)}a_n(\hat p)\,dm_{\mathbb P}(\hat p)
\le
\int_{\mathbb P(\RR^r)}
\limsup_{n\to\infty}a_n(\hat p)\,dm_{\mathbb P}(\hat p).
\]
By Proposition~\ref{lem:annealed-env-FK-upper},
\(\limsup_{n\to\infty}a_n(\hat p)\le\beta\) for every \(\hat p\).  Combining the
last three displays proves the proposition.
\end{proof}

\begin{remark}[Projective averaging versus coordinate directions]
The averaging step cannot be replaced by finitely many fixed coordinate
directions: in general, \(\int\max_i f_i\le\max_i\int f_i\) is false.
Lemma~\ref{lem:projective-averaging} supplies the averaged comparison, and the
uniform bound \(a_n(\hat p)\le M_{\mathcal A}\) justifies the reverse Fatou step.
\end{remark}

\begin{theorem}[Environment-wise Furstenberg formula]
\label{thm:environment-wise-furstenberg-formula}
Assume the logarithmic integrability condition
\eqref{eq:log-integrability-G}, and assume that the mixed base system
\((\Omega\times G^{\NN}\times\TT^d,\mathfrak F,\rho)\)
is ergodic.  Then
\(
        L_1(\rho)=\beta.
\)
\end{theorem}

\begin{proof}
Kingman's formula under mixed-base ergodicity gives
\[
L_1(\rho)
=
\lim_{n\to\infty}
\int_{\Omega\times G^{\NN}\times\TT^d}
\frac1n\log\|\mathcal A_n(\xi,\theta)\|\,
 d\rho(\omega,\xi,\theta).
\]
Proposition~\ref{prop:annealed-operator-FK-upper} implies \(L_1(\rho)\le\beta\).

For the reverse inequality, fix
\(\eta^\bullet=\{\eta^\omega\}_{\omega\in\Omega}\in\mathscr S_Q\), and define
\[
        d\rho_\eta(\omega,\xi,x)
        =
        d\PP(\omega)\,d\mathbf P^\omega(\xi)\,d\eta^\omega(x)
\]
on \(\Omega\times G^{\NN}\times Q\).  Let
\(\mathfrak T(\omega,\xi,x):=(\Theta\omega,\sigma\xi,\xi_0\cdot x)\).
Stationarity of \(\eta^\bullet\) makes \(\rho_\eta\) invariant under
\(\mathfrak T\).  Indeed, for every bounded measurable \(\Phi\), using
\(\mathbf P^\omega=\nu^\omega\otimes\mathbf P^{\Theta\omega}\),
\eqref{eq:averaged-stationary}, and \(\Theta_*\PP=\PP\),
\[
\begin{aligned}
\int \Phi\circ\mathfrak T\,d\rho_\eta
&=
\int_\Omega\int_G\int_{G^{\NN}}\int_Q
\Phi(\Theta\omega,\zeta,g\cdot x)
\,d\eta^\omega(x)\,d\mathbf P^{\Theta\omega}(\zeta)
\,d\nu^\omega(g)\,d\PP(\omega) \\
&=
\int_\Omega\int_{G^{\NN}}\int_Q
\Phi(\Theta\omega,\zeta,y)
\,d\eta^{\Theta\omega}(y)\,d\mathbf P^{\Theta\omega}(\zeta)\,d\PP(\omega) \\
&=
\int_\Omega\int_{G^{\NN}}\int_Q
\Phi(\omega,\zeta,y)
\,d\eta^\omega(y)\,d\mathbf P^\omega(\zeta)\,d\PP(\omega)
=
\int\Phi\,d\rho_\eta .
\end{aligned}
\]

By \eqref{eq:Psi-two-sided-bound}, \(\Psi(\xi_0,x)\in L^1(\rho_\eta)\), and
\[
\int \Psi(\xi_0,x)\,d\rho_\eta(\omega,\xi,x)
=
\int_\Omega\int_Q
\left(\int_G\Psi(g,x)\,d\nu^\omega(g)\right)
 d\eta^\omega(x)\,d\PP(\omega)
=
\alpha(\eta^\bullet).
\]
For \(x=(\theta,\hat p)\in Q\), let \(p\in\hat p\) be a Euclidean unit
representative, set \(X_0=x\), and define \(X_{j+1}:=\xi_j\cdot X_j\).  The
projective cocycle identity gives
\[
        \sum_{j=0}^{n-1}\Psi(\xi_j,X_j)
        =
        \log\|\mathcal A_n(\xi,\theta)p\|.
\]
Since \(\rho_\eta\) is \(\mathfrak T\)-invariant,
\[
\begin{aligned}
\alpha(\eta^\bullet)
&=
\frac1n
\int \sum_{j=0}^{n-1}\Psi(\xi_j,X_j)\,d\rho_\eta \\
&=
\frac1n\int\log\|\mathcal A_n(\xi,\theta)p\|\,
 d\rho_\eta(\omega,\xi,x) \\
&\le
\frac1n\int\log\|\mathcal A_n(\xi,\theta)\|\,
 d\rho_\eta(\omega,\xi,x).
\end{aligned}
\]

By Theorem~\ref{prop:Haar-marginal}, \(\eta^\omega\) has Haar torus
marginal \(m\) for \(\PP\)-a.e. \(\omega\).  Hence the
\((\omega,\xi,\theta)\)-marginal of \(\rho_\eta\) is exactly \(\rho\), and so
\[
\alpha(\eta^\bullet)
\le
\frac1n
\int_{\Omega\times G^{\NN}\times\TT^d}
\log\|\mathcal A_n(\xi,\theta)\|\,
 d\rho(\omega,\xi,\theta)
\qquad(n\ge1).
\]
Letting \(n\to\infty\) and using Kingman's formula gives
\(\alpha(\eta^\bullet)\le L_1(\rho)\).  Since
\(\eta^\bullet\in\mathscr S_Q\) was arbitrary, \(\beta\le L_1(\rho)\).  Together
with \(L_1(\rho)\le\beta\), this proves \(L_1(\rho)=\beta\).
\end{proof}

The formula just proved is the input for the zero-exponent argument in
Section~\ref{sec:exact-equivariance}: when \(L_1(\rho)=0\), maximizing stationary
projective families have vanishing Furstenberg functional, and the next section
upgrades the resulting averaged stationarity to exact equivariance.

\section{Zero-Exponent Exact Equivariance, Degeneracy, and Positivity}
\label{sec:exact-equivariance}

This section is the bridge from the variational formula to the positivity
criterion.  Under mixed-base ergodicity and logarithmic integrability we prove
\[
\begin{aligned}
        L_1(\rho)=0
        &\Longrightarrow
        \text{exact equivariance}\\
        &\Longrightarrow
        \begin{gathered}
        \text{environment-wise compact reduction or}\\
        \text{finite support-invariant subbundle family}.
        \end{gathered}
\end{aligned}
\]
The contrapositive gives positivity once both degeneracies are excluded.  The
first implication uses the Furstenberg formula from
Section~\ref{sec:stationary-families}.  We regularize the one-step cocycle laws only by left
multiplication with small constant matrices, obtain uniformly positive
projective densities, identify the smoothed Furstenberg functional with entropy
production, and let the smoothing vanish.  The second implication is geometric:
exact equivariance propagates the positive-mass projective subspaces over each
environment--phase base point.  A proper common support dimension gives a finite
support-invariant subbundle family, whereas the full-dimensional case gives an
invariant field of inner products by Tyler's barycenter.

\subsection{From zero exponent to exact equivariance}

We prove Theorem~\ref{thm:random-cdk-exactness}.  The kernel
\(\omega\mapsto\nu^\omega\) and the closed supports \(\supp\nu^\omega\subset G\)
are the fixed versions chosen in Section~\ref{sec:model}.  Measurable projective
families are represented as Young-measure disintegrations
\(\lambda_\eta(d\omega,dx)=d\PP(\omega)d\eta^\omega(x),\ x\in Q\).
All pointwise identities below are read after intersecting the full-measure sets
associated with a fixed countable convergence-determining subset of \(C(Q)\).
No measurability of the set-valued map \(\omega\mapsto\supp\nu^\omega\) is used.

Set
\(\mathsf H:=\SL_r(\RR),\  \mathfrak m_Q:=m\times m_{\mathbb P}\),
where \(m_{\mathbb P}\) is the Euclidean projective probability measure from
Lemma~\ref{lem:projective-averaging}.  Embed \(\mathsf H\) into \(G\) by
\(\iota(h)=(0,\underline h), \ \underline h(\theta)\equiv h\),
and write \((0,h)\) for \(\iota(h)\).  Thus
\((0,h)\circ(\alpha,A)=(\alpha,hA), \ (hA)(\theta)=hA(\theta)\).
For \(\lambda\in\Prob(\mathsf H)\) and \(\nu\in\Prob(G)\), define the left
constant-matrix convolution
\[
        \int_G F(k)\,d(\lambda\star\nu)(k)
        =
        \int_{\mathsf H}\int_G
        F((0,h)\circ g)\,d\nu(g)d\lambda(h)
\]
for bounded Borel \(F\).  For smoothing measures \(\lambda_n\) we write
\(\nu_n^\omega:=\lambda_n\star\nu^\omega\).
This changes only the projective matrix direction and leaves the torus
translation marginal unchanged.

\begin{lemma}[Projective smoothing by constant matrices]
\label{lem:compactly-supported-smoothing-kernel}
There exist a compact set \(K\subset\mathsf H\), a relatively compact open set
\(U\Subset\mathsf H\), and probabilities \(\lambda_n\in\Prob_c(\mathsf H)\) such
that
\[
        d\lambda_n(h)=\rho_n(h)\,dh,
        \qquad
        \rho_n\in C_c^\infty(U),
        \qquad
        \supp\lambda_n\subset K,
        \qquad
        \lambda_n\Rightarrow\delta_I .
\]
For each \(n\) there are constants \(0<c_n\le C_n<\infty\) and a continuous
kernel
\(K_n: \mathbb P(\RR^r)\times\mathbb P(\RR^r) \longrightarrow (0,\infty)\)
such that, for every \(\hat p\in\mathbb P(\RR^r)\),
\begin{equation}\label{eq:dirac-projective-density-bound}
        \int_{\mathsf H}h_*\delta_{\hat p}\,d\lambda_n(h)
        =K_n(\hat p,\hat q)\,dm_{\mathbb P}(\hat q),
        \qquad
        c_n\le K_n(\hat p,\hat q)\le C_n .
\end{equation}
Consequently, for every \(\zeta\in\Prob(\mathbb P(\RR^r))\),
\begin{equation}\label{eq:projective-density-bound-lambda-n}
        c_n m_{\mathbb P}
        \le
        \int_{\mathsf H}h_*\zeta\,d\lambda_n(h)
        \le
        C_n m_{\mathbb P} .
\end{equation}
The constants may depend on \(n\), \(r\), and \(K\), but not on \(\zeta\).
\end{lemma}

\begin{proof}
For \(r=1\) take \(\lambda_n=\delta_I\).  Assume \(r\ge2\).  Equip
\(\mathsf H\) with a left-invariant Riemannian metric and scale its Riemannian
volume to be the fixed left Haar measure \(dh\).  Equip
\(\mathbb P(\RR^r)\) with its canonical Riemannian metric, scaled so that its
volume is \(m_{\mathbb P}\).  Choose \(U\Subset\mathsf H\) containing
\(SO(r)\), and choose a smooth probability \(\sigma\) supported in \(U\) whose
density is strictly positive on a neighborhood \(V\Subset U\) of \(SO(r)\).
Let \(\widetilde\lambda_n\) be smooth probabilities supported in an exponential
neighborhood of \(I\), with supports shrinking to \(I\).  With
\(\varepsilon_n\downarrow0\), put
\begin{equation}\label{eq:lambda-n-mixture-definition}
        \lambda_n=(1-\varepsilon_n)\widetilde\lambda_n+
        \varepsilon_n\sigma .
\end{equation}
After enlarging one compact set \(K\Subset U\), all supports lie in \(K\), and
\(\lambda_n\Rightarrow\delta_I\).

For fixed \(\hat p\), the map
\(
        \Phi_{\hat p}:\mathsf H\to\mathbb P(\RR^r),
        \Phi_{\hat p}(h)=h\hat p
\)
is a submersion.  Indeed, at \(h\), with \(v=hp\ne0\), the differential along
\(t\mapsto\exp(tX)h\) sends \(X\in\mathfrak{sl}_r(\RR)\) to the class of \(Xv\) in
\(\RR^r/\RR v\), and every such class is realized by a trace-zero matrix.
The manifold coarea formula
\cite[Appendix, pp.~59--61, formula~(A-2)]{Howard93} applies because
\(\Phi_{\hat p}\) is a smooth submersion and \(\rho_n\) is smooth with compact
support.  It gives the density
\[
        K_n(\hat p,\hat q)
        =
        \int_{\Phi_{\hat p}^{-1}(\hat q)}
        \frac{\rho_n(h)}{J_{\Phi_{\hat p}}(h)}
        \,d\mathcal H_{\hat p}^{\hat q}(h),
\]
where \(J_{\Phi_{\hat p}}>0\) is the normal Jacobian and
\(\mathcal H_{\hat p}^{\hat q}\) is the induced Riemannian volume on the fiber.
The coarea formula alone does not supply the required parameter continuity.
For this, use the simultaneous smooth submersion
\[
        \Psi:\mathbb P(\RR^r)\times\mathsf H
        \longrightarrow
        \mathbb P(\RR^r)\times\mathbb P(\RR^r),
        \qquad
        \Psi(\hat p,h)=(\hat p,h\hat p).
\]
Cover the compact set \(\mathbb P(\RR^r)\times K\) by finitely many submersion
charts and choose a smooth partition of unity subordinate to this cover.  In
each chart the corresponding fiber integral is an integral over fixed
Euclidean fiber variables of a continuous compactly supported integrand.
Dominated convergence gives joint continuity in \((\hat p,\hat q)\); summing the
finitely many chart contributions shows that \(K_n\) is continuous and bounded
above.

For the lower bound, let \(\kappa_0\) be the same density for \(\sigma\).  Since
\(SO(r)\) acts transitively on \(\mathbb P(\RR^r)\), every fiber
\(\Phi_{\hat p}^{-1}(\hat q)\) meets \(SO(r)\).  At such an intersection point
the density of \(\sigma\) is positive on an ambient neighborhood, whose
intersection with the smooth fiber is a relatively open set of positive fiber
volume.  Hence \(\kappa_0(\hat p,\hat q)>0\) for every pair.  Its continuity and
compactness of
\(\mathbb P(\RR^r)\times\mathbb P(\RR^r)\) give a uniform lower bound
\(a>0\).  The mixture
\eqref{eq:lambda-n-mixture-definition} gives
\(K_n\ge\varepsilon_n a\), proving \eqref{eq:dirac-projective-density-bound}.
Averaging \eqref{eq:dirac-projective-density-bound} against any
\(\zeta\in\Prob(\mathbb P(\RR^r))\) gives
\eqref{eq:projective-density-bound-lambda-n}.
\end{proof}

\begin{proposition}[Density bounds for smoothed stationary projective families]
\label{cor:smoothed-stationary-density-bounds}
Let \(\lambda_n\) be as in Lemma~\ref{lem:compactly-supported-smoothing-kernel}
and set \(\nu_n^\omega=\lambda_n\star\nu^\omega\).  Stationary projective
families for \(\nu_n^\omega\) exist.  If the mixed base system is ergodic, then
any such stationary projective family \(\eta_n^\bullet\) satisfies
\begin{equation}\label{eq:smoothed-density-two-sided}
        c_n\mathfrak m_Q
        \le
        \eta_n^\omega
        \le
        C_n\mathfrak m_Q
        \qquad
        \text{for }\PP\text{-a.e. }\omega .
\end{equation}
\end{proposition}

\begin{proof}
The kernel \(\omega\mapsto\nu_n^\omega\) is measurable, so
Proposition~\ref{prop:stationary-existence} gives stationary projective families.  Since
\(\mathfrak a((0,h)\circ g)=\mathfrak a(g)\), the smoothed and unsmoothed kernels
have the same translation marginal.  Hence the smoothed mixed base is ergodic
whenever the original one is, and Theorem~\ref{prop:Haar-marginal} gives
\((\pi_\Sigma)_*\eta_n^\omega=m\).

Disintegrate
\(\eta_n^\omega(d\theta,d\hat p) =dm(\theta)\eta_{n,\theta}^\omega(d\hat p)\).
Projective-fiber stationarity for \(\nu_n^\omega\) gives, for \(\PP\)-a.e. \(\omega\) and
\(m\)-a.e. \(\theta\),
\begin{equation}\label{eq:smoothed-fiber-stationarity-projective}
        \eta_{n,\theta}^{\Theta\omega}
        =
        \int_G\int_{\mathsf H}
        \bigl(h\mathcal A(g)(\theta-\mathfrak a(g))\bigr)_*
        \eta_{n,\theta-\mathfrak a(g)}^\omega
        \,d\lambda_n(h)d\nu^\omega(g).
\end{equation}
Applying \eqref{eq:projective-density-bound-lambda-n} to the inner projective
probability measure and averaging in \(g\) yields
\(
        c_n m_{\mathbb P}\le\eta_{n,\theta}^{\Theta\omega}\le C_n m_{\mathbb P}
\)
for a.e. \(\theta\).  Since \(\Theta_*\PP=\PP\), the same holds with \(\omega\) in
place of \(\Theta\omega\).  Multiplying by \(dm\) proves
\eqref{eq:smoothed-density-two-sided}.
\end{proof}

For the compact set \(K\) above put
\(C_K:= \sup_{h\in K}\bigl(\log^+\|h\|+ \log^+\|h^{-1}\|\bigr)<\infty\).
The standing logarithmic integrability condition and the inverse-norm estimate
following \eqref{eq:log-integrability-G} imply that, for \(h\in K\), \(g\in G\),
\(\theta\in\TT^d\), and \(\|p\|=1\),
\begin{equation}\label{eq:uniform-log-domination-smoothed}
        \left|\log\|h\mathcal A(g)(\theta)p\|\right|
        \le
        C_K+C_{\mathrm{inv}}(1+\Lambda_{\mathcal A}(g)).
\end{equation}
For a stationary projective family \(\eta_n^\bullet\) of the \(n\)-th
smoothed one-step cocycle law, define
\begin{equation}\label{eq:alpha-n-smoothed}
\alpha_n(\eta_n^\bullet)
:=
\int_\Omega\int_G\int_{\mathsf H}\int_Q
\sigma\bigl(h\mathcal A(g)(\theta),\hat p\bigr)
\,d\eta_n^\omega(\theta,\hat p)
 d\lambda_n(h)d\nu^\omega(g)d\PP(\omega),
\end{equation}
which is absolutely convergent by
\eqref{eq:uniform-log-domination-smoothed}.

For probability measures \(\mu,\nu\) on \(Q\), write
\[
        \mathrm H(\mu\mid\nu)
        :=
        \begin{cases}
        \displaystyle\int_Q\log\!\left(\frac{d\mu}{d\nu}\right)d\mu,
        & \mu\ll\nu,\\[1ex]
        +\infty,&\text{otherwise},
        \end{cases}
\]
with natural logarithms.

\begin{proposition}[Entropy identity for smoothed one-step cocycle laws]
\label{prop:smoothed-entropy-identity}
Assume \eqref{eq:log-integrability-G}.  Fix \(n\), and let \(\eta_n^\bullet\) be a stationary projective family for
\(\nu_n^\omega\) and satisfy
\eqref{eq:smoothed-density-two-sided}.  For \(k=(0,h)\circ g\),
\begin{equation}\label{eq:entropy-identity-smoothed}
\alpha_n(\eta_n^\bullet)
=
\frac1r
\int_\Omega\int_G\int_{\mathsf H}
        \mathrm H\bigl(\eta_n^\omega\mid k^{-1}_*\eta_n^{\Theta\omega}\bigr)
        \,d\lambda_n(h)d\nu^\omega(g)d\PP(\omega).
\end{equation}
In particular \(\alpha_n(\eta_n^\bullet)\ge0\).  Equivalently, the entropy defect
\begin{equation}\label{eq:entropy-defect-definition}
        \mathcal D_n
        :=
\int_\Omega\int_G\int_{\mathsf H}
        \mathrm H\bigl(\eta_n^\omega\mid k^{-1}_*\eta_n^{\Theta\omega}\bigr)
        \,d\lambda_n(h)d\nu^\omega(g)d\PP(\omega)
\end{equation}
is finite and satisfies \(\mathcal D_n=r\alpha_n(\eta_n^\bullet)\).
\end{proposition}

\begin{proof}
After modifying on a \(\PP\)-null set, write
\(
        d\eta_n^\omega=f_n^\omega\,d\mathfrak m_Q,
        \ 
        c_n\le f_n^\omega\le C_n,
\)
for both \(\omega\) and \(\Theta\omega\) on the full-measure set used below, with
\((\omega,x)\mapsto f_n^\omega(x)\) measurable.

For \(B\in\SL_r(\RR)\), the Euclidean projective Jacobian formula
\cite[Lemma~3.15]{CDK23} gives
\begin{equation}\label{eq:euclidean-projective-jacobian}
        \frac{d(B^{-1}_*m_{\mathbb P})}{dm_{\mathbb P}}(\hat p)
        =
        \|Bp\|^{-r},
        \qquad \|p\|=1 .
\end{equation}
Here \(B=h\mathcal A(g)(\theta)\in\SL_r(\RR)\), and the torus part of \(k\) is a
Haar-preserving translation.  Thus, for \(x=(\theta,\hat p)\),
\begin{equation}\label{eq:Jk-euclidean}
\begin{aligned}
        J_k(x)
        &:=
        \frac{d(k^{-1}_*\mathfrak m_Q)}{d\mathfrak m_Q}(x)
        =\|h\mathcal A(g)(\theta)p\|^{-r},\\
        \sigma(h\mathcal A(g)(\theta),\hat p)
        &=-\frac1r\log J_k(x).
\end{aligned}
\end{equation}
The density bounds give mutual absolute continuity and
\begin{equation}\label{eq:rn-ratio-smoothed}
        R_{n,\omega,g,h}(x)
        :=
        \frac{d(k^{-1}_*\eta_n^{\Theta\omega})}{d\eta_n^\omega}(x)
        =
        J_k(x)\frac{f_n^{\Theta\omega}(kx)}{f_n^\omega(x)} .
\end{equation}
Taking logarithms in \eqref{eq:rn-ratio-smoothed}, integrating in
\((\omega,g,h,x)\), and using \eqref{eq:uniform-log-domination-smoothed} for
integrability, the two density terms cancel by smoothed stationarity:
\[
\begin{aligned}
&\int \log f_n^{\Theta\omega}(kx)
\,d\eta_n^\omega(x)d\lambda_n(h)d\nu^\omega(g)d\PP(\omega)  \\
&\qquad =
\int_\Omega\int_Q \log f_n^{\Theta\omega}(y)
\,d\eta_n^{\Theta\omega}(y)d\PP(\omega)
=
\int_\Omega\int_Q \log f_n^\omega(y)
\,d\eta_n^\omega(y)d\PP(\omega).
\end{aligned}
\]
Consequently
\[
        \int \log J_k(x)\,d\eta_n^\omega d\lambda_n d\nu^\omega d\PP
        =
        \int \log R_{n,\omega,g,h}(x)\,d\eta_n^\omega d\lambda_n d\nu^\omega d\PP .
\]
By \eqref{eq:Jk-euclidean}, the left-hand side is
\(-r\alpha_n(\eta_n^\bullet)\).  The right-hand side is
\(-\mathcal D_n\), because
\[
        \mathrm H(\eta_n^\omega\mid k^{-1}_*\eta_n^{\Theta\omega})
        =
        -\int_Q\log R_{n,\omega,g,h}\,d\eta_n^\omega .
\]
This proves \eqref{eq:entropy-identity-smoothed}; non-negativity of relative
entropy gives \(\alpha_n\ge0\).
\end{proof}

\begin{proposition}[Removing smoothing in stable limits]
\label{prop:stable-limit-smoothed-families}
Assume \eqref{eq:log-integrability-G}.  Let \(\lambda_n\) and
\(\nu_n^\omega=\lambda_n\star\nu^\omega\) be as above, and let
\(\eta_n^\bullet\) be a stationary projective family for
\(\nu_n^\omega\).  If, after passing to a
subsequence,
\(\eta_n^\bullet\longrightarrow\eta^\bullet\)
stably in \(\mathscr Y_{\PP}(Q)\), then
\(
        \eta^\bullet\in\mathscr S_Q,
        \     \alpha_n(\eta_n^\bullet)\longrightarrow\alpha(\eta^\bullet).
\)
\end{proposition}

\begin{proof}
For \(\varphi\in C(Q)\), define
\[
\begin{aligned}
        P_n^\omega\varphi(x)
        &:=
        \int_G\int_{\mathsf H}\varphi(((0,h)\circ g)x)
        \,d\lambda_n(h)d\nu^\omega(g),\\
        P^\omega\varphi(x)
        &:=
        \int_G\varphi(gx)\,d\nu^\omega(g).
\end{aligned}
\]
If
\(
        \delta_\varphi(h)=\sup_{y\in Q}|\varphi((0,h)y)-\varphi(y)|,
\)
then \(\delta_\varphi\) is bounded and continuous on \(K\), with
\(\delta_\varphi(I)=0\).  Hence \(P_n^\omega\varphi\to P^\omega\varphi\) in
\(L^1(\Omega;C(Q))\), uniformly in \(\omega\).

Smoothed stationarity says
\(
        \int\varphi\,d\eta_n^{\Theta\omega}
        =
        \int P_n^\omega\varphi\,d\eta_n^\omega
\)
for a.e. \(\omega\).  Multiplying by arbitrary \(u\in L^\infty(\PP)\), changing
variables on the left by \(\Theta\), and passing to the stable limit gives
\[
        \int_\Omega u(\omega)
        \left[
        \int_Q\varphi\,d\eta^{\Theta\omega}
        -
        \int_QP^\omega\varphi\,d\eta^\omega
        \right]d\PP(\omega)=0 .
\]
Intersecting over a countable convergence-determining subset of \(C(Q)\) yields
\(
        \eta^{\Theta\omega}=\int_Gg_*\eta^\omega\,d\nu^\omega(g)
\)
for \(\PP\)-a.e. \(\omega\); hence \(\eta^\bullet\in\mathscr S_Q\).

It remains to pass the functional to the limit.  Let
\[
        \Psi_M=\max\{-M,\min\{\Psi,M\}\}
\]
and set
\[
\begin{aligned}
        F_{n,M}(\omega,x)
        &:=
        \int_G\int_{\mathsf H}\Psi_M((0,h)\circ g,x)
        \,d\lambda_n(h)d\nu^\omega(g),\\
        F_M(\omega,x)
        &:=
        \int_G\Psi_M(g,x)\,d\nu^\omega(g).
\end{aligned}
\]
For \(h\in K\),
\(
        |\sigma(hB,\hat p)-\sigma(B,\hat p)|
        \le
        D(h):=\log^+\|h\|+\log^+\|h^{-1}\|,
\)
where \(D\) is bounded on \(K\) and \(D(I)=0\).  Thus
\(F_{n,M}\to F_M\) in \(L^1(\Omega;C(Q))\), and stable convergence gives
\[
        \int F_{n,M}\,d\eta_n^\omega d\PP
        \longrightarrow
        \int F_M\,d\eta^\omega d\PP
        \qquad(M\text{ fixed}).
\]
The truncation is uniform in \(n\) because
\[
        \sup_{h\in K}\sup_{x\in Q}|\Psi((0,h)\circ g,x)|
        \le
        C_K+C_{\mathrm{inv}}(1+\Lambda_{\mathcal A}(g))
        \in L^1(\PP(d\omega)\nu^\omega(dg)).
\]
Letting \(M\to\infty\) proves
\(\alpha_n(\eta_n^\bullet)\to\alpha(\eta^\bullet)\).
\end{proof}

\begin{proposition}[Vanishing entropy defect implies almost-everywhere exact transport]
\label{prop:vanishing-entropy-defect-exact-transport}
Assume the setting of Proposition~\ref{prop:stable-limit-smoothed-families}, and
assume in addition that the entropy defects \(\mathcal D_n\) in
\eqref{eq:entropy-defect-definition} are finite and satisfy \(\mathcal D_n\to0\).
If \(\eta_n^\bullet\to\eta^\bullet\) stably, then
\begin{equation}\label{eq:nu-ae-exact-from-defect}
        g_*\eta^\omega=\eta^{\Theta\omega}
        \qquad
        \text{for }\PP(d\omega)\nu^\omega(dg)\text{-a.e. }(\omega,g).
\end{equation}
\end{proposition}

\begin{proof}
Use the normalization
\(
        \|\mu-\nu\|_{\mathrm{TV}}=\frac12\int|d\mu-d\nu|.
\)
Pinsker's inequality, with this total-variation normalization and natural
logarithms \cite[Lemma~2]{Canonne23}, gives, for probabilities
\(\mu\ll\nu\),
\begin{equation}\label{eq:pinsker-used-form}
        \|\mu-\nu\|_{\mathrm{TV}}^2
        \le
        \frac12\mathrm H(\mu\mid\nu).
\end{equation}
For the smoothed measures below, the two-sided density bounds
\eqref{eq:smoothed-density-two-sided} imply mutual absolute continuity and
finite relative entropy, so the inequality applies.  Applying it with
\(
        \mu=\eta_n^\omega,
        \nu=k^{-1}_*\eta_n^{\Theta\omega},
        k=(0,h)\circ g,
\)
and using Cauchy's inequality for the outer probability measure gives
\begin{equation}\label{eq:L1-almost-equivariance-smoothed}
\int_\Omega\int_G\int_{\mathsf H}
\left\|k^{-1}_*\eta_n^{\Theta\omega}-\eta_n^\omega\right\|_{\mathrm{TV}}
\,d\lambda_n(h)d\nu^\omega(g)d\PP(\omega)
\le
\left(\frac{\mathcal D_n}{2}\right)^{1/2}
\longrightarrow0 .
\end{equation}

Fix \(\varphi\in C(Q)\) and bounded Borel \(b\) on \(\Omega\times G\).  Put
\[
\begin{aligned}
I_n^{(1)}&:=
\int_\Omega\int_G\int_{\mathsf H}
        b(\omega,g)
        \int_Q\varphi(k^{-1}y)\,d\eta_n^{\Theta\omega}(y)
        d\lambda_n(h)d\nu^\omega(g)d\PP(\omega),\\
I_n^{(2)}&:=
\int_\Omega\int_G
        b(\omega,g)
        \int_Q\varphi\,d\eta_n^\omega
        d\nu^\omega(g)d\PP(\omega).
\end{aligned}
\]
By \eqref{eq:L1-almost-equivariance-smoothed},
\begin{equation}\label{eq:test-transport-defect-vanishes}
        |I_n^{(1)}-I_n^{(2)}|\to0 .
\end{equation}
Stable convergence immediately gives
\begin{equation}\label{eq:I2-stable-limit-revised}
        I_n^{(2)}\to
        \int_\Omega\int_G b(\omega,g)
        \int_Q\varphi\,d\eta^\omega
        d\nu^\omega(g)d\PP(\omega).
\end{equation}
For \(I_n^{(1)}\), change variables \(\zeta=\Theta\omega\) and define
\[
        T_n(g,y):=\int_{\mathsf H}\varphi(((0,h)\circ g)^{-1}y)\,d\lambda_n(h),
        \qquad
        T(g,y):=\varphi(g^{-1}y).
\]
For fixed \(g\), \(T_n(g,\cdot)\to T(g,\cdot)\) in \(C(Q)\), and the convergence
is dominated by \(\|\varphi\|_\infty\).  Hence
\[
        F_n(\zeta,y):=
        \int_G b(\Theta^{-1}\zeta,g)T_n(g,y)
        \,d\nu^{\Theta^{-1}\zeta}(g)
\]
converges in \(L^1(\Omega;C(Q))\) to the same expression with \(T\) replacing
\(T_n\).  Stable convergence gives
\begin{equation}\label{eq:I1-stable-limit-revised}
        I_n^{(1)}\to
        \int_\Omega\int_G b(\omega,g)
        \int_Q\varphi(g^{-1}y)\,d\eta^{\Theta\omega}(y)
        d\nu^\omega(g)d\PP(\omega).
\end{equation}
Combining \eqref{eq:test-transport-defect-vanishes},
\eqref{eq:I2-stable-limit-revised}, and \eqref{eq:I1-stable-limit-revised}, we
obtain
\[
\int_{\Omega\times G} b(\omega,g)
\left[
        \int_Q\varphi(g^{-1}y)\,d\eta^{\Theta\omega}(y)
        -
        \int_Q\varphi\,d\eta^\omega
\right]d\PP(\omega)d\nu^\omega(g)=0 .
\]
Since \(b\) is arbitrary, the bracket vanishes for
\(\PP(d\omega)\nu^\omega(dg)\)-a.e. \((\omega,g)\).  Intersecting over a
countable convergence-determining subset of \(C(Q)\) gives
\(g^{-1}_*\eta^{\Theta\omega}=\eta^\omega\) a.e., which is
\eqref{eq:nu-ae-exact-from-defect} after pushing forward by \(g\).
\end{proof}

\begin{lemma}[Upgrade from almost-everywhere exact transport to exact equivariance]
\label{lem:support-level-upgrade}
Let \(\eta^\bullet\) be a measurable family in \(\Prob(Q)\).  If
\begin{equation}\label{eq:ae-exact-before-support-upgrade}
        g_*\eta^\omega=\eta^{\Theta\omega}
        \qquad
        \text{for }\PP(d\omega)\nu^\omega(dg)\text{-a.e. }(\omega,g),
\end{equation}
then there is a Borel set \(\Omega_0\subset\Omega\), \(\PP(\Omega_0)=1\), such
that, for every \(\omega\in\Omega_0\) and every \(g\in\supp\nu^\omega\),
\begin{equation}\label{eq:support-level-exact-upgraded}
        g_*\eta^\omega=\eta^{\Theta\omega}.
\end{equation}
\end{lemma}

\begin{proof}
Let \(\{\varphi_j\}_{j\ge1}\subset C(Q)\) be countable and
convergence-determining.  Define
\[
        D_j(\omega,g):=
        \int_Q\varphi_j(gx)\,d\eta^\omega(x)
        -
        \int_Q\varphi_j\,d\eta^{\Theta\omega}.
\]
By Fubini, after intersecting countably many Borel full-measure sets in
\(\omega\), each \(D_j(\omega,\cdot)\) vanishes \(\nu^\omega\)-a.e.  For fixed such
\(\omega\), the map \(g\mapsto D_j(\omega,g)\) is continuous because the action on
the compact space \(Q\) is continuous.  Its zero set is therefore closed and
contains \(\supp\nu^\omega\).  Hence the integrals of all \(\varphi_j\) against
\(g_*\eta^\omega\) and \(\eta^{\Theta\omega}\) agree for every support element
\(g\), and the convergence-determining property gives
\eqref{eq:support-level-exact-upgraded}.
\end{proof}

\begin{theorem}[Zero-exponent exact equivariance]
\label{thm:random-cdk-exactness}
Assume that the mixed base system
\((\Omega\times G^{\NN}\times\TT^d,\mathfrak F,\rho)\)
is ergodic and that \eqref{eq:log-integrability-G} holds.  If \(L_1(\rho)=0\),
then there exist a stationary projective family
\(\eta^\bullet=\{\eta^\omega\}_{\omega\in\Omega}\in\mathscr S_Q\)
and a Borel set \(\Omega_0\subset\Omega\), \(\PP(\Omega_0)=1\), such that, for
every \(\omega\in\Omega_0\) and every \(g\in\supp\nu^\omega\),
\(
        g_*\eta^\omega=\eta^{\Theta\omega}.
\)
Thus \(\eta^\bullet\) is exact equivariant in the sense of
Definition~\ref{def:exactly-equivariant-family}.
\end{theorem}

\begin{proof}
Choose \(\lambda_n\) and \(\nu_n^\omega=\lambda_n\star\nu^\omega\) as above, and
let \(\eta_n^\bullet\) be any stationary projective family for \(\nu_n^\omega\).
By Proposition~\ref{cor:smoothed-stationary-density-bounds}, the density bounds
\eqref{eq:smoothed-density-two-sided} hold; hence
Proposition~\ref{prop:smoothed-entropy-identity} gives
\(
        \alpha_n(\eta_n^\bullet)\ge0,
        \mathcal D_n=r\alpha_n(\eta_n^\bullet).
\)
Compactness of \(\mathscr Y_\PP(Q)\) gives a stably convergent subsequence,
still denoted \(\eta_n^\bullet\), with limit \(\eta^\bullet\).  By
Proposition~\ref{prop:stable-limit-smoothed-families},
\(     \eta^\bullet\in\mathscr S_Q, 
\alpha_n(\eta_n^\bullet)\to\alpha(\eta^\bullet).
\)
The Furstenberg formula, Theorem~\ref{thm:environment-wise-furstenberg-formula},
and \(L_1(\rho)=0\) imply
\(
        \max_{\zeta^\bullet\in\mathscr S_Q}\alpha(\zeta^\bullet)=0.
\)
Thus \(\alpha(\eta^\bullet)\le0\), while the limit of the nonnegative numbers
\(\alpha_n(\eta_n^\bullet)\) is nonnegative.  Hence
\(
        \alpha(\eta^\bullet)=0,   \alpha_n(\eta_n^\bullet)\to0, 
        \mathcal D_n\to0.
\)
Proposition~\ref{prop:vanishing-entropy-defect-exact-transport} gives
\(
        g_*\eta^\omega=\eta^{\Theta\omega}
\)
for \(\PP(d\omega)\nu^\omega(dg)\)-a.e. \((\omega,g)\).  Lemma~\ref{lem:support-level-upgrade}
upgrades this to every \(g\in\supp\nu^\omega\) on a Borel full-measure set.
\end{proof}

\begin{remark}[Exact equivariance does not imply fixed-vector-fiber subgroup invariance]
\label{rem:support-exactness-no-fiberwise-subgroup-action}
Exact equivariance transports projective measures between successive environment
states.  It gives iterated identities
\[
        (g_{n-1}\cdots g_0)_*\eta^\omega
        =
        \eta^{\Theta^n\omega},
        \qquad
        g_j\in\supp\nu^{\Theta^j\omega},
\]
but not invariance under a subgroup acting on the vector fiber over one fixed
environment--phase base point.  Such a fixed-vector-fiber statement requires a return
or comparison mechanism.  This is the role of the environment-matched
equal-total monodromy groups in Section~\ref{sec:return-monodromy}.
\end{remark}

\subsection{Projective support data and finite support-invariant subbundle families}

Exact equivariance is a statement about projective measures.  To extract a
finite geometric obstruction, we record at each environment--phase base point
the smallest dimension in which positive projective mass occurs, the maximal
mass in that dimension, and the finite family of subspaces attaining it.

For \(k=1,\ldots,r\), let \(\mathrm{Gr}_k(\RR^r)\) denote the
Grassmannian of \(k\)-dimensional linear subspaces of \(\RR^r\), and write
\(\mathscr G_k:=\mathrm{Gr}_k(\RR^r), \ \mathscr G:=\bigsqcup_{k=1}^r\mathscr G_k\).
Each \(W\in\mathscr G_k\) determines \(\mathbb P(W)\subset\mathbb P(\RR^r)\), with
\(\mathbb P(\{0\})=\varnothing\).

\begin{lemma}[Upper semicontinuity of subspace mass]
\label{lem:usc-projective-subspace-mass}
For each \(k\), the map
\(\Phi_k(\mu,W):=\mu(\mathbb P(W)), \  \Prob(\mathbb P(\RR^r))\times\mathscr G_k\to[0,1]\),
is upper semicontinuous, hence Borel.
\end{lemma}

\begin{proof}
If \(\mu_n\to\mu\) weakly and \(W_n\to W\) in \(\mathscr G_k\), then
\(\mathbb P(W_n)\to\mathbb P(W)\) in the Hausdorff topology.  Given
\(\varepsilon>0\), choose an open set \(U\supset\mathbb P(W)\) with
\(\mu(U)\le\mu(\mathbb P(W))+\varepsilon\), and then a closed set
\(C\) with \(\mathbb P(W)\subset\operatorname{int} C\subset C\subset U\).  For
large \(n\), \(\mathbb P(W_n)\subset C\).  The Portmanteau theorem gives
\[
        \limsup_n\mu_n(\mathbb P(W_n))
        \le
        \limsup_n\mu_n(C)
        \le
        \mu(C)
        \le
        \mu(\mathbb P(W))+\varepsilon .
\]
Letting \(\varepsilon\downarrow0\) proves the claim.
\end{proof}

For \(\mu\in\Prob(\mathbb P(\RR^r))\), compactness of \(\mathscr G_k\) and the
preceding lemma imply that
\[
        s_k(\mu):=
        \max_{W\in\mathscr G_k}\mu(\mathbb P(W))
\]
is attained.

\begin{definition}[Positive-mass projective support data]
For \(\mu\in\Prob(\mathbb P(\RR^r))\), define
\[
        e(\mu):=
        \min\{k\in\{1,\ldots,r\}:s_k(\mu)>0\},
        \qquad
        \varrho(\mu):=s_{e(\mu)}(\mu),
\]
and
\[
        \mathfrak L(\mu):=
        \{W\in\mathscr G_{e(\mu)}:
        \mu(\mathbb P(W))=\varrho(\mu)\}.
\]
If a projective family \(\eta^\bullet\) has Haar torus marginal and
\(\eta^\omega(d\theta,d\hat p) =dm(\theta)\eta_\theta^\omega(d\hat p)\)
is a jointly measurable Rokhlin disintegration, set
\[
        e(\omega,\theta):=e(\eta_\theta^\omega),
        \qquad
        \varrho(\omega,\theta):=\varrho(\eta_\theta^\omega),
        \qquad
        \mathfrak L(\omega,\theta):=\mathfrak L(\eta_\theta^\omega)
\]
for \((\PP\otimes m)\)-a.e. \((\omega,\theta)\).
\end{definition}

\begin{lemma}[Measurability and finiteness of support data]
\label{lem:support-data-measurability-finiteness}
The maps \(\mu\mapsto s_k(\mu)\), \(\mu\mapsto e(\mu)\), and
\(\mu\mapsto\varrho(\mu)\) are Borel.  For every \(\mu\), the set
\(\mathfrak L(\mu)\) is nonempty and finite.  More precisely, if \(e(\mu)<r\),
then
\[
        \#\mathfrak L(\mu)
        \le
        \left\lfloor\frac1{\varrho(\mu)}\right\rfloor,
\]
whereas if \(e(\mu)=r\), then \(\mathfrak L(\mu)=\{\RR^r\}\).  The graph of
\(\mu\mapsto\mathfrak L(\mu)\) is Borel in
\(\Prob(\mathbb P(\RR^r))\times\mathscr G\).
\end{lemma}

\begin{proof}
By Lemma~\ref{lem:usc-projective-subspace-mass},
\(s_k(\mu)=\max_{W\in\mathscr G_k}\Phi_k(\mu,W)\) is upper semicontinuous.  Since
\(s_r(\mu)=1\), the function \(e\) is Borel from the description
\(\{e(\mu)=k\} = \{s_1=0,\ldots,s_{k-1}=0, \ s_k>0\}\),
and \(\varrho(\mu)=\sum_{k=1}^r\mathbf 1_{\{e=k\}}s_k(\mu)\) is Borel.

Let \(e=e(\mu)\) and \(\varrho=\varrho(\mu)>0\).  The maximum is attained, so
\(\mathfrak L(\mu)\ne\varnothing\).  If \(e=r\), the only element is \(\RR^r\).  If
\(e<r\) and \(W_1,\ldots,W_N\in\mathfrak L(\mu)\) are distinct, then
\(\dim(W_i\cap W_j)<e\) for \(i\ne j\), so minimality of \(e\) gives
\(\mu(\mathbb P(W_i\cap W_j))=0\).  Hence the \(\mathbb P(W_i)\) are
pairwise disjoint modulo \(\mu\)-null sets and \(N\varrho\le1\).

The graph is the finite union over \(k\) of
\[
\{(\mu,W)\in\Prob(\mathbb P(\RR^r))\times\mathscr G_k:
        e(\mu)=k,
        \mu(\mathbb P(W))=\varrho(\mu)\},
\]
which is Borel.
\end{proof}

\begin{proposition}[Measurability of projective support data]
\label{cor:support-data-for-disintegration}
For a jointly measurable disintegration
\((\omega,\theta)\mapsto\eta_\theta^\omega\), the maps
\((\omega,\theta)\mapsto e(\omega,\theta)\) and
\((\omega,\theta)\mapsto\varrho(\omega,\theta)\) are measurable, and
\[
        \{(\omega,\theta,W):W\in\mathfrak L(\omega,\theta)\}
        \subset
        \Omega\times\TT^d\times\mathscr G
\]
is measurable.  The sections \(\mathfrak L(\omega,\theta)\) are nonempty and finite
for \((\PP\otimes m)\)-a.e. \((\omega,\theta)\).
\end{proposition}

\begin{proof}
The disintegration map into \(\Prob(\mathbb P(\RR^r))\) is measurable.  Compose
with the Borel maps in Lemma~\ref{lem:support-data-measurability-finiteness}; the
graph assertion is the inverse image of the Borel graph of
\(\mu\mapsto\mathfrak L(\mu)\).
\end{proof}

\begin{lemma}[Propagation of positive-mass support data]
\label{lem:propagation-positive-mass-support-data}
Let \(\eta^\bullet\) be an exact equivariant projective family with Haar torus
marginal, and use the fixed Haar disintegration above.  Then there is a Borel
set \(\Omega_0\subset\Omega\), \(\PP(\Omega_0)=1\), such that, for every
\(\omega\in\Omega_0\) and every \(g\in\supp\nu^\omega\), the identities
\[
        e(\Theta\omega,\theta+\mathfrak a(g))=e(\omega,\theta),
        \qquad
        \varrho(\Theta\omega,\theta+\mathfrak a(g))=\varrho(\omega,\theta),
\]
and
\[
        \mathfrak L(\Theta\omega,\theta+\mathfrak a(g))
        =
        \mathcal A(g)(\theta)\,\mathfrak L(\omega,\theta)
\]
hold for \(m\)-a.e. \(\theta\), with the exceptional null set allowed to depend
on \((\omega,g)\).
\end{lemma}

\begin{proof}
If \(B\in\SL_r(\RR)\) and \(\mu'=B_*\mu\), then
\(W\mapsto BW\) is a bijection of each \(\mathscr G_k\) and
\(\mu'(\mathbb P(BW))=\mu(\mathbb P(W))\).  Therefore
\[
        s_k(\mu')=s_k(\mu),
        \quad
        e(\mu')=e(\mu),
        \quad
        \varrho(\mu')=\varrho(\mu),
        \quad
        \mathfrak L(\mu')=B\mathfrak L(\mu).
\]
Exact equivariance gives, on one full-measure set of environments,
\(g_*\eta^\omega=\eta^{\Theta\omega}\) for every support element \(g\).  Disintegrating
over the Haar torus marginals gives
\(
        \eta_{\theta+\mathfrak a(g)}^{\Theta\omega}
        =
        \mathcal A(g)(\theta)_*\eta_\theta^\omega
        \ \text{for }m\text{-a.e. }\theta .
\)
The preceding linear-algebra observation proves the displayed propagation
identities.
\end{proof}

\begin{definition}[Measurable subbundle]
A measurable set
\(\mathcal L\in \cF\otimes\mathcal B(\TT^d)\otimes\mathcal B(\RR^r)\)
is a \(k\)-dimensional measurable subbundle of
\(\Omega\times\TT^d\times\RR^r\) if every vector-fiber section
\(\mathcal L_\theta^\omega :=\{v\in\RR^r:(\omega,\theta,v)\in\mathcal L\}\)
is a \(k\)-dimensional linear subspace of \(\RR^r\).
\end{definition}

\begin{proposition}[Finite support-invariant subbundle family from projective support data]
\label{prop:finite-family-from-propagated-support-data}
Assume that the mixed base system is ergodic.  Let \(\eta^\bullet\) be an exact
equivariant projective family with Haar disintegration
\(\eta^\omega(d\theta,d\hat p)=dm(\theta)\eta_\theta^\omega(d\hat p)\),
and let \(e,\varrho,\mathfrak L\) be the associated support data.  Then there are
constants
\(e_*\in\{1,\ldots,r\}, \  \varrho_*\in(0,1], \ q_*\in\NN\),
such that
\[
        e(\omega,\theta)=e_* ,
        \qquad
        \varrho(\omega,\theta)=\varrho_* ,
        \qquad
        \#\mathfrak L(\omega,\theta)=q_*
\]
for \((\PP\otimes m)\)-a.e. \((\omega,\theta)\).  If \(e_*<r\), then there are
\(q_*\) measurable \(e_*\)-dimensional subbundles
\(\mathcal W_1,\ldots,\mathcal W_{q_*} \subset \Omega\times\TT^d\times\RR^r\)
such that
\(
        \mathfrak L(\omega,\theta)
        =
        \{\mathcal W_{1,\theta}^\omega,
        \ldots,
        \mathcal W_{q_*,\theta}^\omega\}
\)
for \((\PP\otimes m)\)-a.e. \((\omega,\theta)\), and there is a Borel set
\(\Omega_1\subset\Omega\), \(\PP(\Omega_1)=1\), such that, for every
\(\omega\in\Omega_1\) and every \(g\in\supp\nu^\omega\),
\[
        \left\{
        \mathcal A(g)(\theta)\mathcal W_{j,\theta}^\omega:
        1\le j\le q_*
        \right\}
        =
        \left\{
        \mathcal W_{j,\theta+\mathfrak a(g)}^{\Theta\omega}:
        1\le j\le q_*
        \right\}
\]
for \(m\)-a.e. \(\theta\).  Hence the subbundles form a finite support-invariant family, and they are
nonzero and proper.
\end{proposition}

\begin{proof}
By Lemma~\ref{lem:propagation-positive-mass-support-data}, the functions
\(e\), \(\varrho\), and the finite-valued map \(\mathfrak L\) propagate under every
support element.  The same is true of \(\#\mathfrak L\), because the matrix
\(\mathcal A(g)(\theta)\) is invertible.

We first note that \((\omega,\theta)\mapsto\#\mathfrak L(\omega,\theta)\) is
measurable.  Let \(X=\Omega\times\TT^d\) and \(F(x)=\mathfrak L(x)\), modified on a
null set so that all sections are nonempty and finite.  The graph of \(F\) is
Borel by Proposition~\ref{cor:support-data-for-disintegration}.  For each \(N\ge1\), let \(R_N\subset X\times\mathscr G^N\) be the Borel
relation consisting of \((x,W_1,\ldots,W_N)\) with pairwise distinct
\(W_j\in F(x)\).  Its vertical sections are finite.  The Lusin--Novikov
uniformization theorem \cite[Theorem~18.10, pp.~123--124]{Kechris95} implies in
particular that the projection of \(R_N\), namely
\(\{x:\#F(x)\ge N\}\), is Borel.  Thus \(\#F\) is measurable.

Regard \(e\), \(\varrho\), and \(\#\mathfrak L\) as observables on the mixed base,
independent of the path coordinate.  The propagation identities make them
\(\mathfrak F\)-invariant for \(\rho\)-a.e. \((\omega,\xi,\theta)\).  Ergodicity
gives constants \(e_*\), \(\varrho_*\), and \(q_*\), with \(q_*\ge1\), as claimed.

Assume \(e_*<r\).  Let \(D\subset X\) be the full-measure Borel set where
\(e=e_*\), \(\varrho=\varrho_*\), and \(\#\mathfrak L=q_*\).  On \(D\), the
Borel graph of \(\mathfrak L\) has exactly \(q_*\) points in every vertical
section and lies in the compact metric Grassmannian \(\mathscr G_{e_*}\), hence
in a standard Borel space.  The Lusin--Novikov theorem
\cite[Theorem~18.10, pp.~123--124]{Kechris95} writes this graph as a countable
union of graphs of Borel partial functions.  Choosing at each point the
least-index partial function that is defined gives a Borel total selector on
\(D\).  Remove its graph and repeat this construction \(q_*\) times.  The result
is a Borel enumeration by pairwise distinct maps
\(W_1,\ldots,W_{q_*}:D\to\mathscr G_{e_*}\).  Extend the maps outside \(D\) by
one fixed element of \(\mathscr G_{e_*}\).  Then
\(
        \mathfrak L(\omega,\theta)
        =
        \{W_1(\omega,\theta),\ldots,W_{q_*}(\omega,\theta)\}
\)
for a.e. \((\omega,\theta)\).  Define
\(\mathcal W_j:= \{(\omega,\theta,v):v\in W_j(\omega,\theta)\}\).
The tautological set
\(\{(W,v)\in\mathscr G_{e_*}\times\RR^r:v\in W\}\) is closed, so each
\(\mathcal W_j\) is the Borel inverse image of this set under
\((\omega,\theta,v)\mapsto(W_j(\omega,\theta),v)\); hence it is a measurable
subbundle.  Intersecting the propagation full-measure set
with the environments for which the selector representation holds in almost
every phase and after one \(\Theta\)-iterate gives \(\Omega_1\).  Substituting the
selectors into the propagation identity proves the displayed permutation
relation for \(m\)-a.e. \(\theta\).  Since \(1\le e_*<r\), the subbundles are
nonzero and proper.
\end{proof}

The two degeneracies excluded by the positivity theorem are now stated in the
same support-level language.

\begin{definition}[Finite support-invariant subbundle families and environment-wise strong irreducibility]
\label{def:environment-wise-strong-irreducibility}
A finite family of nonzero proper measurable subbundles
\(\mathcal L_1, \ldots, \mathcal L_N \subset \Omega\times\TT^d\times\RR^r, \ N\ge1\),
is called a \emph{finite support-invariant family of nonzero proper measurable
subbundles} if there is a Borel full-measure set
\(\Omega_0\subset\Omega\) such that, for every \(\omega\in\Omega_0\) and every
\(g\in\supp\nu^\omega\),
\[
        \mathcal A(g)(\theta)
        \left(\bigcup_{j=1}^N\mathcal L_{j,\theta}^\omega\right)
        =
        \bigcup_{j=1}^N
        \mathcal L_{j,\theta+\mathfrak a(g)}^{\Theta\omega}
\]
for \(m\)-a.e. \(\theta\).  The cocycle is \emph{environment-wise strongly
irreducible} if it admits no such family.
\end{definition}

Let \(\operatorname{Sym}_r^+\) denote the cone of real symmetric
positive-definite \(r\times r\) matrices, and set
\(
\mathcal S_r:=\{S\in\operatorname{Sym}_r^+:\det S=1\}.
\)

\begin{definition}[Environment-wise compact reduction]
\label{def:compact-reduction}
The cocycle admits an \emph{environment-wise compact reduction} if there are a
measurable map
\(H:\Omega\times\TT^d\to\mathcal S_r\)
and a Borel full-measure set \(\Omega_0\subset\Omega\) such that, for every
\(\omega\in\Omega_0\), every \(g\in\supp\nu^\omega\), and \(m\)-a.e. \(\theta\),
\begin{equation}\label{eq:compact-reduction-identity}
        \mathcal A(g)(\theta)^T
        H(\Theta\omega,\theta+\mathfrak a(g))
        \mathcal A(g)(\theta)
        =
        H(\omega,\theta).
\end{equation}
Equivalently, the inner products
\(
        \langle u,v\rangle_{\omega,\theta}:=u^TH(\omega,\theta)v
\)
are transported isometrically by every support element.
\end{definition}

\begin{definition}[Environment-wise noncompactness]
The cocycle is \emph{environment-wise noncompact} if it admits no
environment-wise compact reduction.
\end{definition}

\subsection{Compact reductions and the positivity criterion}

The preceding subsection gives a finite support-invariant subbundle family when
the common support dimension is proper.  In the remaining case the projective-fiber measures charge no proper
projective subspace.  A measurable equivariant barycenter then converts exact
projective transport into a compact reduction.

\begin{definition}[Proper projective measures and fields]
\label{def:proper-projective-measure}
A measure \(\mu\in\Prob(\mathbb P(\RR^r))\) is \emph{proper} if
\(\mu(\mathbb P(W))=0\)
for every proper linear subspace \(W\subsetneq\RR^r\).  Let
\(\Prob_{\mathrm{prop}}(\mathbb P(\RR^r)) \subset \Prob(\mathbb P(\RR^r))\)
denote the proper locus.  A projective family with Haar disintegration
\(
        \eta^\omega(d\theta,d\hat p)=dm(\theta)\eta_\theta^\omega(d\hat p)
\)
is a \emph{proper projective measure field} if
\(
        \eta_\theta^\omega\in
        \Prob_{\mathrm{prop}}(\mathbb P(\RR^r))
\)
for \((\PP\otimes m)\)-a.e. \((\omega,\theta)\).
\end{definition}

\begin{remark}[Properness and the support index]
Using the support data above,
\(\mu\in\Prob_{\mathrm{prop}}(\mathbb P(\RR^r)) \ \Longleftrightarrow\ e(\mu)=r\).
Indeed, positive mass on a proper projective subspace is equivalent to
\(s_k(\mu)>0\) for some \(k<r\).  Hence the proper locus is Borel by
Lemma~\ref{lem:support-data-measurability-finiteness}.
\end{remark}

On \(\mathcal S_r\) use the \(\SL_r(\RR)\)-action
\(B\cdot H:=B^{-T}HB^{-1}\).

\begin{lemma}[Tyler barycenter for proper projective measures]
\label{lem:furstenberg-barycenter-proper-measures}
There exists a Borel map
\[
        \mathcal B:
        \Prob_{\mathrm{prop}}(\mathbb P(\RR^r))
        \longrightarrow
        \mathcal S_r
\]
such that, for every \(B\in\SL_r(\RR)\) and every proper \(\mu\),
\[
        \mathcal B(B_*\mu)
        =
        B\cdot\mathcal B(\mu)
        =
        B^{-T}\mathcal B(\mu)B^{-1}.
\]
\end{lemma}

\begin{proof}
For \(r=1\), take \(\mathcal B\equiv1\).  Assume \(r\ge2\).  Let
\[
        \iota(\hat p)=\frac{pp^T}{p^Tp},
        \qquad p\in\hat p,
\]
and set \(Q_\mu=\iota_*\mu\).  Since every matrix in the support of
\(Q_\mu\) has rank one and trace one, \(Q_\mu(\{0\})=0\).  Properness gives
Condition~0 of \cite[Section~4.2, equation~(4.5)]{DuembgenPaulySchweizer2015}:
for every nonzero proper \(W\),
\[
        Q_\mu\{M:\operatorname{Ran}M\subset W\}
        =
        \mu(\mathbb P(W))
        =0<\frac{\dim W}{r}.
\]
By \cite[Theorem~4.9]{DuembgenPaulySchweizer2015}, the normalized scatter
criterion has a unique determinant-one minimizer \(\Sigma_0(Q_\mu)\).  In Case~0
of that reference, the objective on determinant-one matrices is, up to the
positive factor \(r\),
\(H\mapsto\Phi_\mu(H)\) after the substitution
\(H=\Sigma^{-1}\).  Thus
\(
        H(\mu):=\Sigma_0(Q_\mu)^{-1}
\)
is the unique element of \(\mathcal S_r\) minimizing
\[
        \Phi_\mu(H):=
        \int_{\mathbb P(\RR^r)}
        \log\frac{p^THp}{p^Tp}\,d\mu(\hat p).
\]
The map \(\mu\mapsto Q_\mu\) is weakly continuous because \(\iota\) is
continuous.  The weak continuity theorem for the normalized scatter functional
\cite[Theorem~6.1 and Remark~6.2]{DuembgenPaulySchweizer2015}, followed by
continuity of matrix inversion on \(\mathcal S_r\), therefore makes
\(\mathcal B(\mu):=H(\mu)\) continuous on the proper locus, and in particular
Borel measurable.

For equivariance, compute for \(B\in\SL_r(\RR)\):
\[
\begin{aligned}
\Phi_{B_*\mu}(B^{-T}HB^{-1})
&=
\int
\log\frac{(Bp)^T(B^{-T}HB^{-1})(Bp)}{(Bp)^T(Bp)}
\,d\mu(\hat p) \\
&=
\Phi_\mu(H)+
\int
\log\frac{p^Tp}{(Bp)^T(Bp)}\,d\mu(\hat p).
\end{aligned}
\]
The second term is finite and independent of \(H\).  Thus \(H\) minimizes
\(\Phi_\mu\) if and only if \(B^{-T}HB^{-1}\) minimizes \(\Phi_{B_*\mu}\).  The
uniqueness of the minimizer proves the stated equivariance.
\end{proof}

\begin{proposition}[Proper projective measure fields yield an environment-wise compact reduction]
\label{prop:proper-projective-field-compact-reduction}
If there exists an exact equivariant proper projective measure field
\(\eta^\bullet\), then the cocycle admits an environment-wise compact reduction.
\end{proposition}

\begin{proof}
Write
\(\eta^\omega(d\theta,d\hat p)=dm(\theta)\eta_\theta^\omega(d\hat p)\).
Let \(X_{\mathrm{prop}}\) be the full \((\PP\otimes m)\)-measure set on which
\(\eta_\theta^\omega\) is proper, and define
\[
        H(\omega,\theta):=
        \mathcal B(\eta_\theta^\omega)
        \quad\text{on }X_{\mathrm{prop}},
        \qquad
        H(\omega,\theta):=I
        \quad\text{off }X_{\mathrm{prop}}.
\]
This is measurable.  Choose a full-measure environment set on which exact
equivariance holds and both \((\omega,\theta)\) and
\((\Theta\omega,\theta+\mathfrak a(g))\) lie in \(X_{\mathrm{prop}}\) for a.e.
\(\theta\).  For such \(\omega\), every \(g\in\supp\nu^\omega\), and a.e.
\(\theta\), disintegrated exactness gives
\(
        \eta_{\theta+\mathfrak a(g)}^{\Theta\omega}
        =
        \mathcal A(g)(\theta)_*\eta_\theta^\omega .
\)
By the equivariance of \(\mathcal B\),
\[
\begin{aligned}
H(\Theta\omega,\theta+\mathfrak a(g))
&=
\mathcal B(\eta_{\theta+\mathfrak a(g)}^{\Theta\omega}) \\
&=
\mathcal A(g)(\theta)^{-T}
H(\omega,\theta)
\mathcal A(g)(\theta)^{-1}.
\end{aligned}
\]
This is exactly \eqref{eq:compact-reduction-identity}.
\end{proof}

\begin{theorem}[Degeneracy alternative from exact equivariance]
\label{thm:degeneracy-exact}
Assume that the mixed base system is ergodic and that an exact equivariant
projective family on \(Q\) exists.  Then at least one of the following holds:
\begin{enumerate}
\item the cocycle admits an environment-wise compact reduction;
\item the cocycle admits a finite support-invariant family of nonzero proper
measurable subbundles.
\end{enumerate}
\end{theorem}

\begin{proof}
Let \(\eta^\bullet\) be exact equivariant.  Integrating the support-level identity
against \(\nu^\omega\) gives averaged stationarity, and mixed-base ergodicity
gives the Haar torus marginal by Theorem~\ref{prop:Haar-marginal}.  Fix a Haar
disintegration and the support data \(e,\varrho,\mathfrak L\).  By
Proposition~\ref{prop:finite-family-from-propagated-support-data},
\(e(\omega,\theta)=e_*\) a.e. for some \(e_*\in\{1,\ldots,r\}\).  If \(e_*<r\),
the same proposition gives a finite support-invariant family of nonzero proper
measurable subbundles.

If \(e_*=r\), then the characterization
\(
        \mu\in\Prob_{\mathrm{prop}}(\mathbb P(\RR^r))
        \Longleftrightarrow e(\mu)=r
\)
shows that the projective-fiber measures are proper a.e.  Thus \(\eta^\bullet\) is an exact
equivariant proper projective measure field, and
Proposition~\ref{prop:proper-projective-field-compact-reduction} gives a compact
reduction.
\end{proof}

\begin{theorem}[Environment-wise Furstenberg positivity criterion]
\label{thm:env-positivity}
Assume that the mixed base system
\((\Omega\times G^{\NN}\times\TT^d,\mathfrak F,\rho)\)
is ergodic and that \eqref{eq:log-integrability-G} holds.  If the cocycle is
environment-wise noncompact and environment-wise strongly irreducible, then
\(
        L_1(\rho)>0.
\)
\end{theorem}

\begin{proof}
If \(L_1(\rho)=0\), Theorem~\ref{thm:random-cdk-exactness} supplies an exact
equivariant projective family.  Theorem~\ref{thm:degeneracy-exact} then gives
either an environment-wise compact reduction or a finite support-invariant
family of nonzero proper measurable subbundles, contradicting respectively environment-wise
noncompactness and environment-wise strong irreducibility.  Hence
\(L_1(\rho)\ne0\).

Since every \(B\in\SL_r(\RR)\) satisfies \(\|B\|\ge1\), we have
\(\log\|\mathcal A_n(\xi,\theta)\|\ge0\) for every \(n\).  Kingman's formula
therefore gives \(L_1(\rho)\ge0\).  Combining this with \(L_1(\rho)\ne0\) proves
\(L_1(\rho)>0\).
\end{proof}

\section{Environment-Matched Equal-Total Monodromy Groups and Positivity}
\label{sec:return-monodromy}

This section converts the cocycle-level degeneracy alternative of
Section~\ref{sec:exact-equivariance} into tests based on the moving supports.
An admissible support word is tied to an environment segment and determines a
word map between vector fibers over environment--phase base points.  Two word
maps can be compared only when their words belong to the same forward or
backward admissible support-word space and have equal total torus translation.
The first condition matches the environment segment and the sequence of
supports, and the second gives the same total phase displacement and matches the
remaining endpoint phase.  Their relative composition is an
\emph{environment-matched equal-total word quotient}. Geometrically, it is a
fiberwise return matrix on the vector fiber over the common
\((\omega,\theta)\).  These word quotients generate the
\emph{environment-matched equal-total monodromy group}.  Stable noncompactness
and stable strong irreducibility of a suitable equal-total monodromy group
exclude, respectively, an environment-wise compact reduction and a finite
support-invariant family of nonzero proper measurable subbundles.  The
positivity criterion of Theorem~\ref{thm:env-positivity} then yields
\(L_1(\rho)>0\).

\begin{remark}[Pointwise support convention]
\label{rem:pointwise-support-monodromy}
All support-level constructions in this section are pointwise in the
environment.  After discarding a null set, we work on a fixed
\(\Theta\)-invariant full-measure Borel set on which all kernels
\(\nu^{\Theta^j\omega}\), \(j\in\ZZ\), are defined.  For such a fixed
\(\omega\), the admissible support-word spaces, equal-total word-pair spaces, dense
subsets, and environment-matched equal-total monodromy groups introduced below
are constructed from the closed supports
\(\supp\nu^{\Theta^j\omega}\) along this orbit.  We never use
\(\omega\mapsto\supp\nu^\omega\) as a measurable multifunction, nor
\(\omega\mapsto\cM^{\mathrm{eq}}_{\omega,\theta}\) as a random closed subgroup.

The countable dense subsets used below may therefore depend on the fixed
\(\omega\).  Their only purpose is to pass from a.e.-in-phase support identities
to countably many word-pair identities.  After the corresponding regular phase set of full Haar measure is formed,
continuity in the word-pair variables and closedness of the relevant stabilizers
extend the conclusion to all equal-total word pairs.
\end{remark}

\subsection{Environment-matched equal-total word quotients}

Fix a good environment \(\omega\) as in
Remark~\ref{rem:pointwise-support-monodromy}.  Since \(\Theta\) is invertible,
the supports at \(\Theta^j\omega\), \(j\in\ZZ\), are available.  For
\(\ell\ge1\), define the forward and backward admissible support-word spaces
\[
\mathfrak W_{\ell}^{+}(\omega)
:=
\supp\nu^{\omega}\times
\supp\nu^{\Theta\omega}\times\cdots\times
\supp\nu^{\Theta^{\ell-1}\omega},
\]
and
\[
\mathfrak W_{\ell}^{-}(\omega)
:=
\supp\nu^{\Theta^{-\ell}\omega}\times
\supp\nu^{\Theta^{-(\ell-1)}\omega}\times\cdots\times
\supp\nu^{\Theta^{-1}\omega}.
\]
Elements of \(\mathfrak W_\ell^+(\omega)\) and
\(\mathfrak W_\ell^-(\omega)\) are called forward and backward admissible
support words, respectively.
For \(\mathbf g=(g_0,\ldots,g_{\ell-1})\in G^\ell\), set
\[
        \Gamma_j(\mathbf g):=\sum_{i=0}^{j-1}\mathfrak a(g_i),
        \qquad
        \Gamma_0(\mathbf g):=0 .
\]
Thus \(\Gamma_\ell(\mathbf g)\) is the total torus translation of the word,
which is also its total phase displacement.

For \(\mathbf g\in\mathfrak W_\ell^+(\omega)\), define the forward word map
\[
\mathcal A_{\mathbf g}^{+}(\theta)
:=
\mathcal A(g_{\ell-1})(\theta+\Gamma_{\ell-1}(\mathbf g))
\cdots
\mathcal A(g_1)(\theta+\Gamma_1(\mathbf g))
\mathcal A(g_0)(\theta),
\]
which maps the vector fiber over \((\omega,\theta)\) to the vector fiber over
\((\Theta^\ell\omega,\theta+\Gamma_\ell(\mathbf g))\).  For
\(\mathbf g\in\mathfrak W_\ell^-(\omega)\), define the backward word map
\[
\begin{aligned}
\mathcal A_{\mathbf g}^{-}(\theta)
&:=
\mathcal A(g_{\ell-1})
(\theta+\Gamma_{\ell-1}(\mathbf g)-\Gamma_\ell(\mathbf g))
\cdots\\
&\qquad
\mathcal A(g_1)
(\theta+\Gamma_1(\mathbf g)-\Gamma_\ell(\mathbf g))
\mathcal A(g_0)
(\theta-\Gamma_\ell(\mathbf g)),
\end{aligned}
\]
which maps the vector fiber over
\((\Theta^{-\ell}\omega,\theta-\Gamma_\ell(\mathbf g))\) to the vector
fiber over \((\omega,\theta)\).

Equal total torus translation is the phase-matching condition: it identifies
the target phase for forward word maps and the source phase for backward word
maps.  Define
\[
\mathfrak P_{\ell}^{+}(\omega)
:=
\left\{
(\mathbf g,\mathbf h)
\in\mathfrak W_{\ell}^{+}(\omega)\times\mathfrak W_{\ell}^{+}(\omega):
\Gamma_\ell(\mathbf g)=\Gamma_\ell(\mathbf h)
\right\}
\]
and
\[
\mathfrak P_{\ell}^{-}(\omega)
:=
\left\{
(\mathbf g,\mathbf h)
\in\mathfrak W_{\ell}^{-}(\omega)\times\mathfrak W_{\ell}^{-}(\omega):
\Gamma_\ell(\mathbf g)=\Gamma_\ell(\mathbf h)
\right\}.
\]
These are the forward and backward equal-total word-pair spaces.  For such
pairs, define the relative compositions
\[
Q_{\mathbf g,\mathbf h}^{+}(\theta)
:=
\bigl(\mathcal A_{\mathbf g}^{+}(\theta)\bigr)^{-1}
\mathcal A_{\mathbf h}^{+}(\theta),
\qquad
Q_{\mathbf g,\mathbf h}^{-}(\theta)
:=
\mathcal A_{\mathbf g}^{-}(\theta)
\bigl(\mathcal A_{\mathbf h}^{-}(\theta)\bigr)^{-1}.
\]
The matrices \(Q_{\mathbf g,\mathbf h}^{\pm}(\theta)\) are the
environment-matched equal-total word quotients; geometrically, each is a
fiberwise return matrix on the vector fiber over \((\omega,\theta)\).  The \(\ell\)-level
environment-matched equal-total word quotient sets are
\[
\mathcal E_{\ell,\omega,\theta}^{\pm}
:=
\left\{
Q_{\mathbf g,\mathbf h}^{\pm}(\theta):
(\mathbf g,\mathbf h)\in\mathfrak P_\ell^{\pm}(\omega)
\right\}
\subset\SL_r(\RR),
\]
and
\(\mathcal E_{\omega,\theta}^{\pm} :=\bigcup_{\ell\ge1}\mathcal E_{\ell,\omega,\theta}^{\pm}\).

\begin{definition}[Environment-matched equal-total monodromy group]
The \emph{environment-matched equal-total monodromy group} at \((\omega,\theta)\) is
\[
\cM^{\mathrm{eq}}_{\omega,\theta}
:=
\overline{
\left\langle
\mathcal E_{\omega,\theta}^{+}\cup
\mathcal E_{\omega,\theta}^{-}
\right\rangle_{\mathrm{grp}}
}^{\,\SL_r(\RR)}.
\]
Here equal-total means equal total torus translation, equivalently equal
total phase displacement.  It is a closed subgroup of \(\SL_r(\RR)\).  After
this definition, we use \emph{equal-total monodromy group} as a short form when
the environment matching is clear.
\end{definition}

\begin{remark}[Elementary one-step word quotients]
\label{rem:elementary-generators}
If \((\beta,B_1),(\beta,B_2)\in\supp\nu^\omega\), then
\[
        B_1(\theta)^{-1}B_2(\theta)
        \in\mathcal E^+_{1,\omega,\theta}
        \subset\cM^{\mathrm{eq}}_{\omega,\theta}.
\]
If \((\beta,B_1),(\beta,B_2)\in\supp\nu^{\Theta^{-1}\omega}\), then
\[
        B_1(\theta-\beta)B_2(\theta-\beta)^{-1}
        \in\mathcal E^-_{1,\omega,\theta}
        \subset\cM^{\mathrm{eq}}_{\omega,\theta}.
\]
A \emph{monodromy witness} is a concrete word quotient, or a finite set of
word quotients, used to verify a group-level property of the equal-total
monodromy group.  These one-step word quotients serve as monodromy witnesses
in the Schr\"odinger application.
\end{remark}

\subsection{Stable subgroup nondegeneracy in the lower-Chabauty topology}

The monodromy-group criterion uses only the lower half of the Chabauty--Fell
topology: compact pieces already present in a subgroup must persist under
perturbation.  No upper semicontinuity of the whole subgroup is required.

\begin{definition}[Lower-Chabauty topology]
\label{def:lower-chabauty}
Let
\[
\operatorname{Sub}(\SL_r(\RR))
:=
\{H\subset\SL_r(\RR): H \text{ is a closed subgroup}\}.
\]
Fix a compatible metric \(d\) on \(\SL_r(\RR)\), and let
\(d_{\mathrm H}\) be the induced Hausdorff distance on compact subsets.  For
\(H\in\operatorname{Sub}(\SL_r(\RR))\), a compact set \(K\subset H\), and \(\delta>0\),
set
\[
\mathcal U_{K,\delta}(H)
:=
\left\{
H'\in\operatorname{Sub}(\SL_r(\RR)):
\exists K'\subset H' \text{ compact with }
        d_{\mathrm H}(K,K')<\delta
\right\}.
\]
The \emph{lower-Chabauty topology} is the topology for which these sets form a
neighborhood basis of \(H\), as \(K\) and \(\delta\) vary.  It is the lower part
of the usual Chabauty--Fell topology \cite{Chabauty50,Fell62}.
\end{definition}

\begin{definition}[Stable subgroup nondegeneracy]
\label{def:stable-subgroup-nondegeneracy}
Let \(H\in\operatorname{Sub}(\SL_r(\RR))\).
\begin{enumerate}[label=(\roman*)]
    \item \(H\) is \emph{strongly irreducible} if it preserves no finite union of
    nonzero proper linear subspaces of \(\RR^r\).
    \item \(H\) is \emph{stably noncompact} if some lower-Chabauty neighborhood
    of \(H\) consists only of noncompact closed subgroups.
    \item \(H\) is \emph{stably strongly irreducible} if some lower-Chabauty
    neighborhood of \(H\) consists only of strongly irreducible closed subgroups.
\end{enumerate}
We call the properties in items \textup{(ii)} and \textup{(iii)}
\emph{stable noncompactness} and \emph{stable strong irreducibility},
respectively.
\end{definition}

\begin{lemma}[Lower-Chabauty continuity in the phase variable]
\label{lem:phase-lower-chabauty-monodromy}
There exists a \(\Theta\)-invariant full-measure Borel set
\(\Omega_{\mathrm{eq}}\subset\Omega\) such that, for every
\(\omega\in\Omega_{\mathrm{eq}}\), the map
\[
        \TT^d\ni\theta\longmapsto
        \cM^{\mathrm{eq}}_{\omega,\theta}
        \in\operatorname{Sub}(\SL_r(\RR))
\]
is lower-Chabauty continuous.  Equivalently, if
\(K\subset\cM^{\mathrm{eq}}_{\omega,\theta_0}\) is compact and \(\delta>0\),
then for all \(\theta\) sufficiently close to \(\theta_0\) there is a compact
\(K_\theta\subset\cM^{\mathrm{eq}}_{\omega,\theta}\) with
\(d_{\mathrm H}(K,K_\theta)<\delta\).
\end{lemma}

\begin{proof}
Take \(\Omega_{\mathrm{eq}}\) to be the invariant full-measure set on which the
pointwise support convention of Remark~\ref{rem:pointwise-support-monodromy} is
valid.  Fix \(\omega\in\Omega_{\mathrm{eq}}\).  The equal-total condition and the
support membership of a word pair are independent of \(\theta\).  For each fixed
\((\mathbf g,\mathbf h)\in\mathfrak P_\ell^\pm(\omega)\), the quotient
\(Q_{\mathbf g,\mathbf h}^{\pm}(\theta)\) depends continuously on \(\theta\), and
so does every finite product of such quotients and their inverses.

Let \(K\subset\cM^{\mathrm{eq}}_{\omega,\theta_0}\) be compact and let
\(\delta>0\).  Choose \(x_1,\ldots,x_N\in K\) forming a \(\delta/3\)-net.  Since
\(\cM^{\mathrm{eq}}_{\omega,\theta_0}\) is the closed subgroup generated by the
environment-matched equal-total word quotients, for each \(i\) choose a finite
product \(P_i(\theta_0)\) of these word quotients and their inverses
with
\(d(P_i(\theta_0),x_i)<\delta/3\).
Using the same finitely many equal-total word pairs,
define \(P_i(\theta)\) for
nearby phases.  By continuity, after shrinking the phase neighborhood of
\(\theta_0\),
\(d(P_i(\theta),P_i(\theta_0))<\delta/3, \  i=1,\ldots,N\).
Then
\(K_\theta:=\{P_1(\theta),\ldots,P_N(\theta)\} \subset\cM^{\mathrm{eq}}_{\omega,\theta}\)
is compact and satisfies \(d_{\mathrm H}(K,K_\theta)<\delta\).  This is lower-Chabauty continuity.
\end{proof}

\subsection{Monodromy-group exclusions and positivity}

We now show that the two cocycle-level degeneracies from
Section~\ref{sec:exact-equivariance} impose incompatible restrictions on
stable equal-total monodromy groups.  An environment-wise compact reduction
places each regular-phase group in a compact inner-product stabilizer, whereas
a finite support-invariant family of nonzero proper measurable subbundles
places it in the stabilizer of a finite union of proper subspaces.

\begin{lemma}[Inner-product stabilizers]
\label{lem:closed-stabilizer-inner-product}
Let \(S\in\mathcal S_r\) and set
\(K_S:=\{M\in\SL_r(\RR): M^T S M=S\}\).
Then \(K_S\) is a closed compact subgroup of \(\SL_r(\RR)\).
\end{lemma}

\begin{proof}
Closedness is immediate.  If \(R=S^{1/2}\), then
\[
M^T S M=S
\quad\Longleftrightarrow\quad
(RMR^{-1})^T(RMR^{-1})=I.
\]
For \(M\in\SL_r(\RR)\), this identifies \(K_S\) with \(R^{-1}SO(r)R\), hence it
is compact.
\end{proof}

\begin{lemma}[Finite-union stabilizers]
\label{lem:closed-stabilizer-finite-union}
Let \(\mathcal L=\{L_1,\ldots,L_N\}\) be a finite family of nonzero proper
linear subspaces of \(\RR^r\), and set
\[
\operatorname{Stab}(\mathcal L)
:=
\left\{
M\in\SL_r(\RR):
M\left(\bigcup_{j=1}^N L_j\right)=\bigcup_{j=1}^N L_j
\right\}.
\]
Then \(\operatorname{Stab}(\mathcal L)\) is a closed subgroup of
\(\SL_r(\RR)\).
\end{lemma}

\begin{proof}
It is a subgroup.  If \(M_n\in\operatorname{Stab}(\mathcal L)\) and
\(M_n\to M\), then the finite union \(\bigcup_j L_j\) is closed, so
\(
        M\left(\bigcup_j L_j\right)\subset\bigcup_j L_j .
\)
Applying the same argument to \(M_n^{-1}\to M^{-1}\) gives the reverse
inclusion.  Thus \(M\in\operatorname{Stab}(\mathcal L)\).
\end{proof}

For the two obstruction arguments, fix a good \(\omega\).  Since \(G\) is
Polish, the supports are closed, and \(\Gamma_\ell:G^\ell\to\TT^d\) is
continuous, each \(\mathfrak P_\ell^\pm(\omega)\) is a closed separable subspace
of \(G^\ell\times G^\ell\).  Choose countable dense subsets
\(
        \mathfrak D_\ell^+(\omega)\subset\mathfrak P_\ell^+(\omega),
        \ 
        \mathfrak D_\ell^-(\omega)\subset\mathfrak P_\ell^-(\omega).
\)
These choices are pointwise in \(\omega\).  Countability is used only to form
intersections of sets of full Haar measure for the a.e. phase identities below; continuity of
\(Q_{\mathbf g,\mathbf h}^{\pm}(\theta)\) in the word-pair variables and
closedness of the relevant stabilizer then extend the conclusion to all
\(\mathfrak P_\ell^\pm(\omega)\).

\begin{proposition}[Environment-wise compact reductions force compact equal-total monodromy groups]
\label{prop:compact-reduction-controls-return-monodromy}
Assume that the cocycle admits an environment-wise compact reduction
\(H:\Omega\times\TT^d\longrightarrow\mathcal S_r\).
Then for \(\PP\)-a.e. \(\omega\) there exists a set of full Haar measure
\(\mathcal R_H(\omega)\subset\TT^d\) such that, for every
\(\theta\in\mathcal R_H(\omega)\),
\(
        \cM^{\mathrm{eq}}_{\omega,\theta}
        \subset K_{H(\omega,\theta)} .
\)
In particular, \(\cM^{\mathrm{eq}}_{\omega,\theta}\) is compact at every such
phase.
\end{proposition}

\begin{proof}
Let \(\Omega_H^*\) be the \(\Theta\)-invariant full-measure core obtained by
intersecting the compact-reduction domain over all \(\Theta\)-iterates and with
\(\Omega_{\mathrm{eq}}\) from
Lemma~\ref{lem:phase-lower-chabauty-monodromy}.  Fix
\(\omega\in\Omega_H^*\) and choose the dense sets
\(\mathfrak D_\ell^\pm(\omega)\) above.

Iterating \eqref{eq:compact-reduction-identity} along a fixed forward word gives,
for \(m\)-a.e. \(\theta\),
\[
(\mathcal A_{\mathbf g}^{+}(\theta))^T
H(\Theta^\ell\omega,\theta+\Gamma_\ell(\mathbf g))
\mathcal A_{\mathbf g}^{+}(\theta)
=H(\omega,\theta).
\]
Along a fixed backward word it gives, for \(m\)-a.e. \(\theta\),
\[
(\mathcal A_{\mathbf g}^{-}(\theta))^T
H(\omega,\theta)
\mathcal A_{\mathbf g}^{-}(\theta)
=H(\Theta^{-\ell}\omega,\theta-\Gamma_\ell(\mathbf g)).
\]
Intersect the corresponding phase sets of full Haar measure over all component words that
occur in the countable dense pairs, and call the resulting set
\(\mathcal R_H(\omega)\).

Fix \(\theta\in\mathcal R_H(\omega)\) and write \(S:=H(\omega,\theta)\).  For a
forward dense pair \((\mathbf g,\mathbf h)\in\mathfrak D_\ell^+(\omega)\), put
\[
        A=\mathcal A_{\mathbf g}^{+}(\theta),
        \qquad
        B=\mathcal A_{\mathbf h}^{+}(\theta),
        \qquad
        T=H(\Theta^\ell\omega,\theta+\Gamma_\ell(\mathbf g)).
\]
The common forward admissible support-word space gives the same source and target environments,
and the two word maps are evaluated at the same phase \(\theta\), so they
share their source vector fiber.  The equal-total condition identifies their
target phase, so they also share their target vector fiber.  Hence \(A^TTA=S\) and \(B^TTB=S\).  Therefore
\(
        (A^{-1}B)^T S(A^{-1}B)=S.
\)
Thus every dense forward environment-matched equal-total word quotient lies in
\(K_S\).

For a backward dense pair, with
\(A=\mathcal A_{\mathbf g}^{-}(\theta)\) and
\(B=\mathcal A_{\mathbf h}^{-}(\theta)\), the common backward admissible support-word space
gives the same source and target environments, and both word maps are evaluated
at the same phase \(\theta\), so they share their target vector fiber.  The
equal-total condition identifies their source phase, so they also share their
source vector fiber.  Thus the identities have the form
\(A^TSA=T\) and \(B^TSB=T\) with the same source value \(T\).  Hence
\(
        (AB^{-1})^T S(AB^{-1})=S,
\)
so every dense backward environment-matched equal-total word quotient also lies in
\(K_S\).

By continuity in the word-pair variables and closedness of \(K_S\), every
environment-matched equal-total word quotient in
\(\mathcal E_{\omega,\theta}^+\cup\mathcal E_{\omega,\theta}^-\) lies
in \(K_S\).  Since \(K_S\) is a closed subgroup,
\(\cM^{\mathrm{eq}}_{\omega,\theta}\subset K_S\).
Lemma~\ref{lem:closed-stabilizer-inner-product} gives compactness of \(K_S\).
\end{proof}

\begin{proposition}[Finite support-invariant subbundle families constrain equal-total monodromy groups to finite-union stabilizers]
\label{prop:finite-family-controls-return-monodromy}
Assume that
\(\mathcal L_1,\ldots,\mathcal L_N
\subset\Omega\times\TT^d\times\RR^r\)
form a finite support-invariant family of nonzero proper measurable subbundles.  Then for \(\PP\)-a.e.
\(\omega\) there exists a set of full Haar measure
\(\mathcal R_{\mathcal L}(\omega)\subset\TT^d\) such that, for every
\(\theta\in\mathcal R_{\mathcal L}(\omega)\),
\(
\cM^{\mathrm{eq}}_{\omega,\theta}
\subset
\operatorname{Stab}
\bigl(
\{\mathcal L_{1,\theta}^{\omega},\ldots,
\mathcal L_{N,\theta}^{\omega}\}
\bigr).
\)
Consequently, \(\cM^{\mathrm{eq}}_{\omega,\theta}\) preserves the finite union
\(\bigcup_{j=1}^N\mathcal L_{j,\theta}^{\omega}\) and is not strongly
irreducible.
\end{proposition}

\begin{proof}
Let \(\Omega_{\mathcal L}^*\) be the \(\Theta\)-invariant full-measure core
obtained by intersecting the support-invariance domain of the subbundle family over all
\(\Theta\)-iterates and with \(\Omega_{\mathrm{eq}}\).  Fix
\(\omega\in\Omega_{\mathcal L}^*\) and choose the
dense sets \(\mathfrak D_\ell^\pm(\omega)\).

For each fixed admissible support word, iterating the support-invariance relation
gives, for
\(m\)-a.e. \(\theta\),
\[
\mathcal A_{\mathbf g}^{+}(\theta)
\left(\bigcup_{j=1}^N\mathcal L_{j,\theta}^{\omega}\right)
=
\bigcup_{j=1}^N
\mathcal L_{j,\theta+\Gamma_\ell(\mathbf g)}^{\Theta^\ell\omega},
\]
and
\[
\mathcal A_{\mathbf g}^{-}(\theta)
\left(\bigcup_{j=1}^N
\mathcal L_{j,\theta-\Gamma_\ell(\mathbf g)}^{\Theta^{-\ell}\omega}\right)
=
\bigcup_{j=1}^N\mathcal L_{j,\theta}^{\omega}.
\]
Intersect the phase sets of full Haar measure over the component words in all dense pairs,
and also the phase set where the subbundle fibers are nonzero and proper; denote the
result by \(\mathcal R_{\mathcal L}(\omega)\).

Fix \(\theta\in\mathcal R_{\mathcal L}(\omega)\), and set
\[
        U_{\omega,\theta}:=\bigcup_{j=1}^N\mathcal L_{j,\theta}^{\omega},
        \qquad
        \mathscr L_{\omega,\theta}:=
        \{\mathcal L_{1,\theta}^{\omega},\ldots,
        \mathcal L_{N,\theta}^{\omega}\}.
\]
The stabilizer \(\operatorname{Stab}(\mathscr L_{\omega,\theta})\) is closed by
Lemma~\ref{lem:closed-stabilizer-finite-union}.

For a forward dense pair \((\mathbf g,\mathbf h)\in\mathfrak D_\ell^+(\omega)\),
the common forward admissible support-word space gives the same source and target environments,
and the two word maps are evaluated at the same phase \(\theta\), so they
share their source phase.  The equal-total condition identifies the target
phase.  Hence the target union is the same for \(\mathbf g\) and
\(\mathbf h\), and
\(
        (\mathcal A_{\mathbf g}^{+}(\theta))^{-1}
        \mathcal A_{\mathbf h}^{+}(\theta)U_{\omega,\theta}
        =U_{\omega,\theta}.
\)
For a backward dense pair, the common backward admissible support-word space gives the same
source and target environments, and the two word maps are evaluated at the same
phase \(\theta\), so they share their target phase.  The equal-total condition
identifies the source phase.  Hence the source union is the same, and
\(
        \mathcal A_{\mathbf g}^{-}(\theta)
        (\mathcal A_{\mathbf h}^{-}(\theta))^{-1}U_{\omega,\theta}
        =U_{\omega,\theta}.
\)
Thus all dense environment-matched equal-total word quotients lie in
\(\operatorname{Stab}(\mathscr L_{\omega,\theta})\).  Continuity in the word-pair
variables and closedness of the stabilizer extend this to all
environment-matched equal-total word quotients, and taking the closed
generated subgroup gives
\(
\cM^{\mathrm{eq}}_{\omega,\theta}
\subset
\operatorname{Stab}(\mathscr L_{\omega,\theta}).
\)
Since \(U_{\omega,\theta}\) is a finite union of nonzero proper subspaces, this
environment-matched equal-total monodromy group is not strongly irreducible.
\end{proof}

\begin{theorem}[Exclusion criterion from environment-matched equal-total monodromy groups]
\label{thm:monodromy-criterion-full}
Assume that there exist Borel sets
\(\Omega_{\mathrm{nc}},\Omega_{\mathrm{si}}\subset\Omega\) with
\(
        \PP(\Omega_{\mathrm{nc}})>0, \PP(\Omega_{\mathrm{si}})>0,
\)
such that:
\begin{enumerate}[label=(M\arabic*)]
    \item for every \(\omega\in\Omega_{\mathrm{nc}}\), there exists a phase
    \(\theta_\omega^{\mathrm{nc}}\in\TT^d\) such that
    \(\cM^{\mathrm{eq}}_{\omega,\theta_\omega^{\mathrm{nc}}}\) is stably
    noncompact;
    \item for every \(\omega\in\Omega_{\mathrm{si}}\), there exists a phase
    \(\theta_\omega^{\mathrm{si}}\in\TT^d\) such that
    \(\cM^{\mathrm{eq}}_{\omega,\theta_\omega^{\mathrm{si}}}\) is stably
    strongly irreducible.
\end{enumerate}
Then the cocycle is
environment-wise noncompact and environment-wise strongly irreducible.
\end{theorem}

\begin{proof}
Suppose first that a compact reduction \(H\) exists.  Let \(\Omega_H^*\) be the
full-measure set from
Proposition~\ref{prop:compact-reduction-controls-return-monodromy}.  Since
\(\Omega_{\mathrm{eq}}\) and \(\Omega_H^*\) have full measure,
\(
\PP\bigl(
\Omega_{\mathrm{nc}}
\cap\Omega_{\mathrm{eq}}
\cap\Omega_H^*
\bigr)
=
\PP(\Omega_{\mathrm{nc}})>0.
\)
Choose
\(
\omega\in
\Omega_{\mathrm{nc}}
\cap\Omega_{\mathrm{eq}}
\cap\Omega_H^*,
\ 
\theta_0:=\theta_\omega^{\mathrm{nc}}.
\)
Stable noncompactness gives a lower-Chabauty neighborhood
\(\mathcal O_{\mathrm{nc}}\) of
\(\cM^{\mathrm{eq}}_{\omega,\theta_0}\) consisting only of noncompact closed
subgroups.  By Lemma~\ref{lem:phase-lower-chabauty-monodromy}, for \(\theta\) in
some nonempty open neighborhood \(V\) of \(\theta_0\),
\(\cM^{\mathrm{eq}}_{\omega,\theta}\in\mathcal O_{\mathrm{nc}}\).
Since Haar measure has full support, choose
\(\theta\in V\cap\mathcal R_H(\omega)\).  Then
\(\cM^{\mathrm{eq}}_{\omega,\theta}\) is noncompact, while
Proposition~\ref{prop:compact-reduction-controls-return-monodromy} gives
\(\cM^{\mathrm{eq}}_{\omega,\theta}\subset K_{H(\omega,\theta)}\)
with compact right-hand side.  This contradiction excludes compact reductions;
the cocycle is environment-wise noncompact.

Suppose next that a finite support-invariant family of nonzero proper
measurable subbundles exists.
Let \(\Omega_{\mathcal L}^*\) be the full-measure set from
Proposition~\ref{prop:finite-family-controls-return-monodromy}.  Since
\(\Omega_{\mathrm{eq}}\) and \(\Omega_{\mathcal L}^*\) have full measure,
\(
\PP\bigl(
\Omega_{\mathrm{si}}
\cap\Omega_{\mathrm{eq}}
\cap\Omega_{\mathcal L}^*
\bigr)
=
\PP(\Omega_{\mathrm{si}})>0.
\)
Choose
\(
\omega\in
\Omega_{\mathrm{si}}
\cap\Omega_{\mathrm{eq}}
\cap\Omega_{\mathcal L}^*,\ 
\theta_0:=\theta_\omega^{\mathrm{si}}.
\)
Stable strong irreducibility gives a lower-Chabauty neighborhood
\(\mathcal O_{\mathrm{si}}\) of
\(\cM^{\mathrm{eq}}_{\omega,\theta_0}\) consisting only of strongly irreducible
closed subgroups.  Again Lemma~\ref{lem:phase-lower-chabauty-monodromy} gives a
nonempty open phase neighborhood \(V\) of \(\theta_0\) such that
\(\cM^{\mathrm{eq}}_{\omega,\theta}\in\mathcal O_{\mathrm{si}} \ (\theta\in V)\).
Choose \(\theta\in V\cap\mathcal R_{\mathcal L}(\omega)\).  At this phase the
environment-matched equal-total monodromy group is strongly irreducible by membership in
\(\mathcal O_{\mathrm{si}}\), but
Proposition~\ref{prop:finite-family-controls-return-monodromy} places it in the
stabilizer of a finite union of nonzero proper subspaces.  This contradiction
excludes finite support-invariant subbundle families; the cocycle is
environment-wise strongly
irreducible.

In each contradiction, the putative degeneracy supplies a full-measure regular
environment set, while only one environment from its intersection with the
corresponding environment set is used.  Positive measure of each environment set is
therefore sufficient; mere nonemptiness would not ensure such an intersection.
\end{proof}

\begin{theorem}[Positivity from environment-matched equal-total monodromy groups]
\label{thm:positivity-via-return-monodromy}
Assume that the mixed base system
\((\Omega\times G^{\NN}\times\TT^d,\mathfrak F,\rho)\)
is ergodic and that the logarithmic integrability condition
\eqref{eq:log-integrability-G} holds.  Assume also that there exist Borel sets
\(\Omega_{\mathrm{nc}},\Omega_{\mathrm{si}}\subset\Omega\) with
\(
        \PP(\Omega_{\mathrm{nc}})>0, \PP(\Omega_{\mathrm{si}})>0,
\)
such that:
\begin{enumerate}[label=(M\arabic*)]
    \item for every \(\omega\in\Omega_{\mathrm{nc}}\), there exists a phase
    \(\theta_\omega^{\mathrm{nc}}\in\TT^d\) such that
    \(\cM^{\mathrm{eq}}_{\omega,\theta_\omega^{\mathrm{nc}}}\) is stably
    noncompact;
    \item for every \(\omega\in\Omega_{\mathrm{si}}\), there exists a phase
    \(\theta_\omega^{\mathrm{si}}\in\TT^d\) such that
    \(\cM^{\mathrm{eq}}_{\omega,\theta_\omega^{\mathrm{si}}}\) is stably
    strongly irreducible.
\end{enumerate}
The two environment sets and the two chosen phases need not coincide,
and no measurability of either phase choice is assumed.  Then
\(
        L_1(\rho)>0.
\)
\end{theorem}

\begin{proof}
By Theorem~\ref{thm:monodromy-criterion-full}, the cocycle is
environment-wise noncompact and environment-wise strongly irreducible.  The
mixed-base ergodicity and logarithmic integrability assumptions then allow us
to apply Theorem~\ref{thm:env-positivity}, which yields \(L_1(\rho)>0\).
\end{proof}

\section{Stationarily Varying Random--Quasiperiodic Schr\"odinger Cocycles}
\label{sec:schrodinger}

This final section specializes the moving-law model to stationarily varying
random--quasiperiodic Schr\"odinger cocycles.  We first define the
Schr\"odinger parameter kernel, the induced one-step cocycle laws, and their
frequency marginals.  We then give two mixed-base ergodicity criteria.  Finally,
same-frequency distinct-noise pairs in adjacent noisy environments supply two
one-step word quotients that serve as unipotent monodromy witnesses.
Theorem~\ref{thm:positivity-via-return-monodromy} then gives positivity at
every energy.  For standard background on one-dimensional Schr\"odinger
transfer matrices and Lyapunov exponents, see \cite{BL85,Damanik17}.

\subsection{Schr\"odinger parameter kernels, one-step cocycle laws, and frequency marginals}
\label{sec:abstract-schrodinger-setup}

Fix a continuous potential \(v\in C^0(\TT^d,\RR)\).  For
\(E\in\RR\), set
\[
S_{E,v}(\theta):=
\begin{pmatrix}
 v(\theta)-E & -1\\
 1 & 0
\end{pmatrix}
\in \SL_2(\RR),
\]
and, for \(\zeta\in\RR\), let
\[
P(\zeta):=
\begin{pmatrix}
 1 & \zeta\\
 0 & 1
\end{pmatrix},
\qquad
A_{E,\zeta}(\theta):=P(\zeta)S_{E,v}(\theta)
=
\begin{pmatrix}
 v(\theta)+\zeta-E & -1\\
 1 & 0
\end{pmatrix}.
\]
Thus \(A_{E,\zeta}\in C^0(\TT^d,\SL_2(\RR))\).  With the present sign convention,
\(A_{E,\zeta_n}(\theta_n)\) is the transfer matrix for the recursion
\[
        \binom{u_{n+1}}{u_n}
        =A_{E,\zeta_n}(\theta_n)\binom{u_n}{u_{n-1}},
        \qquad
        u_{n+1}+u_{n-1}+(E-v(\theta_n)-\zeta_n)u_n=0 .
\]
When \(E\) is fixed, we write \(A_\zeta\) for \(A_{E,\zeta}\).

Let \(\omega\mapsto\kappa^\omega\in\Prob_c(\TT^d\times\RR)\) be a measurable
Schr\"odinger parameter kernel.  For each energy \(E\), define the induced Schr\"odinger one-step cocycle law on \(\coc(d,2)\) by
\begin{equation}\label{eq:schrod-law}
\nu_E^\omega
:=
\int_{\TT^d\times\RR}
\delta_{(\alpha,A_{E,\zeta})}
\,d\kappa^\omega(\alpha,\zeta)
\in \Prob_c\bigl(\coc(d,2)\bigr).
\end{equation}
For \(g_{\alpha,\zeta}:=(\alpha,A_{E,\zeta})\), the vector-fiber matrix is
\(\mathcal A(g_{\alpha,\zeta})(\theta) =A_{E,\zeta}(\theta) =P(\zeta)S_{E,v}(\theta)\).
Let \(\rho_E\) be the mixed invariant measure associated with \(\nu_E^\bullet\).  We call
\(L_1(E):=L_1(\rho_E)\) the \emph{top Lyapunov exponent at energy \(E\)}.

Throughout this section, after changing the kernel on a null set if necessary,
we assume a uniform noise-support bound: there is \(R>0\) such that
\begin{equation}\label{eq:schrod-bounded-support}
        \supp\kappa^\omega\subset \TT^d\times[-R,R]
        \qquad
        \text{for all }\omega\in\Omega .
\end{equation}
Since \(v\) is continuous and \(\zeta\) is uniformly bounded, the matrices
\(A_{E,\zeta}(\theta)\) are uniformly bounded in \((\zeta,\theta)\) for each fixed
\(E\).  Hence the logarithmic integrability condition
\eqref{eq:log-integrability-G} holds for the one-step cocycle laws \(\nu_E^\omega\).

Let \(\pi_{\TT^d}:\TT^d\times\RR\to\TT^d\) be the first-coordinate projection and
set
\(\mu^\omega:=(\pi_{\TT^d})_*\kappa^\omega\).
This frequency marginal is precisely the translation marginal of the
Schr\"odinger one-step cocycle law \(\nu_E^\omega\).  It determines the
translation factor because the variable \(\zeta\) changes only the vector-fiber
matrix.

\subsection{Mixed-base ergodicity criteria for frequency marginals}

The next two criteria provide the mixed-base ergodicity input for the Schr\"odinger
application.  They are stated directly for \(\mu^\omega\), because
Theorem~\ref{thm:ergodicity-random-environment} depends on a one-step cocycle law only through its translation marginal.

\begin{proposition}[Mixed-base ergodicity from Fourier contraction]
\label{prop:random-env-support-ergodicity}
Suppose that, for every \(k\in\ZZ^d\setminus\{0\}\),
\(\PP\Bigl\{ \omega: |\widehat{\mu^\omega}(k)|<1 \Bigr\}>0\).
Then the mixed base system associated with the translation marginals
\(\mu^\omega\) is ergodic.
\end{proposition}

\begin{proof}
Fix \(k\in\ZZ^d\setminus\{0\}\), put
\(c(\omega):=\widehat{\mu^\omega}(k)\), and suppose
\(u(\omega)=c(\omega)u(\Theta\omega) \ \PP\text{-a.e.}\)
with \(u\in L^\infty(\Omega,\PP)\).  Since \(|c|\le1\),
\(|u(\omega)|\le |u(\Theta\omega)|\).  The two sides have the same \(L^2\)-norm,
because \(\Theta\) preserves \(\PP\), and therefore
\(|u(\omega)|=|u(\Theta\omega)| \ \PP\text{-a.e.}\)
The set \(B:=\{\omega:|u(\omega)|>0\}\) is \(\Theta\)-invariant modulo null sets.
If \(u\neq0\), ergodicity of \(\Theta\) gives \(\PP(B)=1\), and the identity
\(|u(\omega)|=|c(\omega)|\,|u(\Theta\omega)|\)
forces \(|c(\omega)|=1\) for \(\PP\)-a.e. \(\omega\), contradicting the hypothesis
for this \(k\).  Hence \(u=0\).  Since this holds for every nonzero \(k\),
Theorem~\ref{thm:ergodicity-random-environment}, item \textup{(4)}, applies.
\end{proof}

\begin{proposition}[Mixed-base ergodicity from a common rationally independent support frequency]
\label{prop:weak-mixing-fixed-support-frequency}
Assume that \(\Theta\) is weakly mixing.  Let \(\alpha_*\in\TT^d\) be rationally
independent, i.e.
\(e_k(\alpha_*)\neq1 \ \forall k\in\ZZ^d\setminus\{0\}\),
and suppose that
\(\alpha_*\in\supp\mu^\omega \ \text{for }\PP\text{-a.e. }\omega\).
Then the mixed base system associated with the translation marginals
\(\mu^\omega\) is ergodic.
\end{proposition}

\begin{proof}
Fix \(k\in\ZZ^d\setminus\{0\}\), set
\(c(\omega):=\widehat{\mu^\omega}(k)\), and consider an obstruction
\(u\in L^\infty(\Omega,\PP)\) satisfying
\(u(\omega)=c(\omega)u(\Theta\omega)\).  If \(\PP\{|c|<1\}>0\), the argument in the proof of
Proposition~\ref{prop:random-env-support-ergodicity} gives \(u=0\).  It remains
to consider the case \(|c|=1\) a.e.  Since \(|e_k|=1\), equality in the triangle
inequality implies that \(e_k\) is constant \(\mu^\omega\)-a.e.; by continuity it is
constant on \(\supp\mu^\omega\), for a.e. \(\omega\).  The common support frequency \(\alpha_*\) then gives
\(c(\omega)=e_k(\alpha_*) \  \PP\text{-a.e.}\)
The obstruction equation becomes
\(u\circ\Theta=\overline{e_k(\alpha_*)}\,u\).
Rational independence gives \(\overline{e_k(\alpha_*)}\neq1\), and weak mixing
excludes nonzero Koopman eigenfunctions with this eigenvalue.  Thus \(u=0\).
Theorem~\ref{thm:ergodicity-random-environment}, item \textup{(4)}, now gives
mixed-base ergodicity.
\end{proof}

\subsection{Same-frequency distinct-noise monodromy witnesses}

We now verify the group-level hypotheses in dimension two.  A same-frequency
distinct-noise pair at \(\omega\) gives a forward environment-matched
equal-total word quotient that is lower unipotent, while such a pair at
\(\Theta^{-1}\omega\) gives a backward word quotient that is upper unipotent.
Geometrically, these quotients are fiberwise return matrices, and together they
form a pair of monodromy witnesses.  Their generated group contains a
hyperbolic element, and a two-element lower-Chabauty criterion gives the stable
noncompactness and stable strong irreducibility required by
Theorem~\ref{thm:positivity-via-return-monodromy}.

\begin{lemma}[Two-element criterion for stable subgroup nondegeneracy in \(\SL_2\)]
\label{lem:stable-sl2-witness}
Let \(H\in\operatorname{Sub}(\SL_2(\RR))\).  Suppose that \(H\) contains elements
\(W,T\) such that \(W\) is hyperbolic and, writing
\(F_W\subset\mathbb P^1\) for its two fixed projective lines, the projective map
\(T\) preserves no nonempty subset of \(F_W\).  Then \(H\) is stably noncompact
and stably strongly irreducible in the lower-Chabauty topology.
\end{lemma}

\begin{proof}
Fix the compatible group metric \(d_G\) used in
Definition~\ref{def:lower-chabauty}, and fix a compatible metric
\(d_{\mathbb P}\) on \(\mathbb P^1\).  On the hyperbolic locus, the unordered
two-point set
\(A\mapsto F_A\) varies continuously in the Hausdorff metric
\(d_{\mathrm H}^{\mathbb P}\).

The assumption on \(T\) means that \(T\) fixes neither point of \(F_W\) and does
not map \(F_W\) onto itself.  Hence
\[
\Delta:=
\min\left\{
        \min_{x\in F_W} d_{\mathbb P}(Tx,x),
        d_{\mathrm H}^{\mathbb P}(T(F_W),F_W)
\right\}>0.
\]
By continuity, there are neighborhoods \(\mathcal N_W\) of \(W\) and
\(\mathcal N_T\) of \(T\) such that every \(W'\in\mathcal N_W\) is hyperbolic and
\(T'\) preserves no nonempty subset of \(F_{W'}\) whenever
\((W',T')\in\mathcal N_W\times\mathcal N_T\).
Choose \(\delta>0\) with
the \(d_G\)-balls \(B_G(W,\delta)\subset\mathcal N_W\) and
\(B_G(T,\delta)\subset\mathcal N_T\), set \(K:=\{W,T\}\subset H\), and let
\(\mathcal U:=\mathcal U_{K,\delta}(H)\) be the lower-Chabauty neighborhood from
Definition~\ref{def:lower-chabauty}.  If \(H'\in\mathcal U\), then \(H'\) contains
\(W',T'\) with
\(d_G(W',W)<\delta, \  d_G(T',T)<\delta\).
Thus \(W'\) is hyperbolic and \(T'\) preserves no nonempty subset of \(F_{W'}\).
The powers of \(W'\) are unbounded, so \(H'\) is noncompact.

Suppose that \(H'\) preserved a nonempty finite set
\(F\subset\mathbb P^1\).  Since \(W'F=F\), and a hyperbolic projective map has no
finite orbits outside its two fixed lines, one has \(F\subset F_{W'}\).  But
\(T'F=F\), so \(T'\) preserves a nonempty subset of \(F_{W'}\), a contradiction.
Therefore \(H'\) preserves no finite subset of \(\mathbb P^1\), equivalently no
finite union of nonzero proper subspaces of \(\RR^2\).  Hence \(H'\) is strongly
irreducible.  The same neighborhood \(\mathcal U\) gives both stability
conclusions.
\end{proof}

\begin{proposition}[Positivity from same-frequency distinct-noise pairs]
\label{prop:schrodinger-random-frequency-noise-same-frequency}
Assume that \(\Theta\) is weakly mixing and that there exists a rationally
independent \(\alpha_*\in\TT^d\) such that
\(\alpha_*\in\supp\mu^\omega \ \text{for }\PP\text{-a.e. }\omega\).
Suppose also that there is a Borel set \(\Omega_{\mathrm{sf}}\subset\Omega\)
such that, for every \(\omega\in\Omega_{\mathrm{sf}}\), the support of
\(\kappa^\omega\) contains the same-frequency distinct-noise pair
\(
(\alpha_\omega,\zeta_{1,\omega}),
(\alpha_\omega,\zeta_{2,\omega}),
\zeta_{1,\omega}\neq\zeta_{2,\omega},
\)
where \(\alpha_\omega\) is the environment-dependent comparison frequency and
is not assumed to equal the common rationally independent support frequency
\(\alpha_*\).  Assume in addition the adjacent noisy-environment overlap
condition
\(
\PP\left(
\Omega_{\mathrm{sf}}
\cap
\{\omega\in\Omega:\Theta^{-1}\omega\in\Omega_{\mathrm{sf}}\}
\right)>0.
\)
Then, for every
energy \(E\in\RR\),
\(
L_1(E)>0.
\)
\end{proposition}

\begin{proof}
Fix \(E\in\RR\) and write \(A_\zeta=A_{E,\zeta}\).  By
Proposition~\ref{prop:weak-mixing-fixed-support-frequency}, the mixed base for
\(\nu_E^\omega\) is ergodic.  The support bound
\eqref{eq:schrod-bounded-support} gives the logarithmic integrability condition
\eqref{eq:log-integrability-G}.  It remains only to verify the stable noncompactness and stable strong
irreducibility hypotheses for the equal-total monodromy groups in
Theorem~\ref{thm:positivity-via-return-monodromy}.

Let \(\Omega_{\mathrm{eq}}\) be the full-measure set from
Lemma~\ref{lem:phase-lower-chabauty-monodromy}, and set
\[
        \Omega_0
        :=
        \Omega_{\mathrm{sf}}
        \cap
        \{\omega\in\Omega:\Theta^{-1}\omega\in\Omega_{\mathrm{sf}}\}
        \cap
        \Omega_{\mathrm{eq}}.
\]
This is a Borel set and
\(
\PP(\Omega_0)>0,
\)
because \(\Omega_{\mathrm{eq}}\) has full measure.  Fix \(\omega\in\Omega_0\).
The current environment \(\omega\) supplies a same-frequency distinct-noise pair
\(
        (\alpha,\zeta_1),(\alpha,\zeta_2)\in\supp\kappa^\omega,
        \ \zeta_1\neq\zeta_2,
\)
and the preceding environment \(\Theta^{-1}\omega\) supplies another such pair
\(
        (\alpha_-,\zeta^-_1),(\alpha_-,\zeta^-_2)
        \in\supp\kappa^{\Theta^{-1}\omega},
        \ \zeta^-_1\neq\zeta^-_2.
\)
By continuity of \((\alpha,\zeta)\mapsto(\alpha,A_{E,\zeta})\), the corresponding cocycle elements lie in
\(\supp\nu_E^\omega\) and \(\supp\nu_E^{\Theta^{-1}\omega}\).  The elementary one-step word quotients in
Remark~\ref{rem:elementary-generators} therefore give,
for every \(\theta\in\TT^d\),
\[
        A_{\zeta_1}(\theta)^{-1}A_{\zeta_2}(\theta)
        =
        \begin{pmatrix}1&0\\ \zeta_1-\zeta_2&1\end{pmatrix}
        \in\cM^{\mathrm{eq}}_{\omega,\theta},
\]
and
\[
        A_{\zeta^-_1}(\theta-\alpha_-)
        A_{\zeta^-_2}(\theta-\alpha_-)^{-1}
        =
        \begin{pmatrix}1&\zeta^-_1-\zeta^-_2\\ 0&1\end{pmatrix}
        \in\cM^{\mathrm{eq}}_{\omega,\theta}.
\]
These two word quotients are fiberwise return matrices and, together, form
a pair of monodromy witnesses.  Put \(\delta:=\zeta_1-\zeta_2\) and
\(u:=\zeta^-_1-\zeta^-_2\).  Then
\(\delta u\neq0\), and the group \(\cM^{\mathrm{eq}}_{\omega,\theta}\) contains
\[
        L:=\begin{pmatrix}1&0\\ \delta&1\end{pmatrix},
        \qquad
        U:=\begin{pmatrix}1&u\\ 0&1\end{pmatrix}.
\]
Choose \(\varepsilon\in\{1,-1\}\) so that \(\varepsilon u\delta>0\), and set
\(W:=U^\varepsilon L\).
Since \(\cM^{\mathrm{eq}}_{\omega,\theta}\) is a group, it contains \(W\).  Moreover
\(\operatorname{tr}W=2+\varepsilon u\delta=2+|u\delta|>2\),
so \(W\) is hyperbolic.  The lower unipotent \(L\) has unique fixed projective
line \([e_2]\), while
\(W e_2=U^\varepsilon L e_2=\varepsilon u e_1+e_2\)
is not a multiple of \(e_2\).  Hence \([e_2]\notin F_W\).  A nontrivial
parabolic projective transformation has no finite projective orbit outside its
unique fixed line, so \(L\) preserves no nonempty subset of \(F_W\).
Lemma~\ref{lem:stable-sl2-witness}, applied with
\(H=\cM^{\mathrm{eq}}_{\omega,\theta}\) and \(T=L\), shows that
\(\cM^{\mathrm{eq}}_{\omega,\theta}\) is stably noncompact and stably strongly
irreducible.

The preceding argument holds for every \(\theta\in\TT^d\) and every
\(\omega\in\Omega_0\).  Thus the stable noncompactness and stable strong irreducibility hypotheses
for the equal-total monodromy groups in
Theorem~\ref{thm:positivity-via-return-monodromy} hold with
\(
        \Omega_{\mathrm{nc}}=\Omega_{\mathrm{si}}=\Omega_0, \theta_\omega^{\mathrm{nc}}
        =\theta_\omega^{\mathrm{si}}
        =0_{\TT^d}.
\)
That theorem gives \(L_1(E)>0\).  Since \(E\) was arbitrary, the conclusion
holds for all energies.
\end{proof}

\begin{remark}[Relation with the fixed-law criteria of Cai--Duarte--Klein]
 The same-frequency distinct-noise criterion uses environment-matched
equal-total counterparts
of the one-step unipotent monodromy matrices used in
\cite[Propositions~4.1 and~4.4]{CDK23}.  If
\(\Omega=\{\ast\}\), \(\Theta=\mathrm{id}\), and
\(\kappa^\omega\equiv\eta\in\Prob_c(\TT^d\times\RR)\), the model specializes to
a fixed-law one and the environment-matching requirement becomes vacuous.  If \(\supp\eta\) contains
\((\alpha,\zeta_1)\) and \((\alpha,\zeta_2)\), with \(\alpha\) rationally
independent and \(\zeta_1\neq\zeta_2\), the mixed-base ergodicity input is supplied by
\(\alpha_*:=\alpha\), and the same one-step unipotent matrices are obtained.
Under the hypotheses of
\cite[Proposition~4.4]{CDK23}, these matrices reproduce its
Lyapunov-positivity argument.  When the frequency marginal is the single atom
\(\delta_\alpha\), the same matrices likewise reproduce the fixed-frequency
random-noise
positivity argument of \cite[Proposition~4.1]{CDK23}.  
\end{remark}

\end{document}